\documentclass[10pt, fleqn, reqno]{article}
\usepackage{fixmath}

\usepackage[T1]{fontenc}
\usepackage{ucs}
\usepackage[utf8x]{inputenc}
\usepackage[english]{babel}

\usepackage{authblk}

\usepackage{tikz}
\usepackage{mathrsfs}

\usepackage{microtype}
\usepackage{babel}

\usepackage{mathtools} 

\usepackage{enumerate}

\usepackage{geometry} 

\usepackage{comment}
\specialcomment{smalltt}
{\begingroup\ttfamily\footnotesize}{\endgroup}
\includecomment{comment}

\usepackage{amsfonts, amsmath, amsthm, amssymb}
\usepackage[colorlinks=true,urlcolor=blue]{hyperref} 
\usepackage[msc-links]{amsrefs} 
\usepackage{doi} 
\renewcommand{\PrintDOI}[1]{DOI \doi{#1}}

\newtheorem{theorem}{Theorem}
\newtheorem{lemma}{Lemma}
\newtheorem{proposition}{Proposition}
\newtheorem{corollary}{Corollary}

\theoremstyle{definition}
\newtheorem{example}{Example}

\theoremstyle{remark}
\newtheorem{remark}{Remark}

\DeclareRobustCommand{\lah}{\genfrac{\lfloor}{\rfloor}{0pt}{}}

\providecommand{\Prob}[1]{\mathbb{P}\left\{#1\right\}}
\providecommand{\abs}[1]{\lvert#1\rvert}
\providecommand{\inorm}[1]{\lVert#1\rVert}
\providecommand{\supnorm}[1]{\lVert#1\rVert_\infty}

\providecommand{\card}[1]{\texttt{\#}#1}
\providecommand{\risfac}[2]{#1^{\overline{#2}}} 
\providecommand{\fallfac}[2]{#1^{\underline{#2}}} 

\providecommand{\gen}[1]{\mathcal{#1}} 

\newcommand{\type}[1]{\mathfrak{c}_{#1}}

\newcommand{\Ex}{\mathbb{E}}

\newcommand{\pset}[1]{\mathcal{P}_{#1}}

\title{Small-time behaviour and hydrodynamic limit\\ of beta coalescents}

\author[1]{Luke Miller\thanks{Department of Statistics, 24--29 St~Giles', Oxford, OX1 3LB, UK;~\href{lmiller@stats.ox.ac.uk}{lmiller@stats.ox.ac.uk}.}
\and Helmut H.~Pitters\thanks{Department of Statistics, 367 Evans Hall, U.C.~Berkeley, CA 94720, USA and Department of Statistics, 24--29 St~Giles', Oxford, OX1 3LB, UK;~\href{helmut.pitters@berkeley.edu}{helmut.pitters@berkeley.edu}.}~\thanks{To whom correspondence should be addressed.}}

\begin{document}
\maketitle

\begin{abstract}
We quantify the behaviour at small times of the beta coalescent $\Pi=\{ \Pi(t), t\geq 0\}$ with parameters $a, b>0$.
To this end we study the non-trivial limits of the rescaled block counting process $\{n^{\alpha}\card\Pi_n(tn^{\beta}), t\geq 0\}$ as $n\to\infty$ for suitable $\alpha$ and $\beta,$ the idea being to approximate $\Pi$ with its restriction to $\{1, \ldots, n\}$, $\Pi_n$. If $\Pi$ comes down from infinity we obtain a Law of Large Numbers type of result, or hydrodynamic limit in the parlance of statistical physics, that exhibits a phase transition at $\alpha=-1$. More specifically, for $\alpha>-1$ and $\beta=\alpha(1-a)$ the rescaling limit $c(t)$ is deterministic
\begin{align*}
  c(t) = \left(\frac{\Gamma(a+b)}{(2-a)\Gamma(b)}t\right)^{\frac{1}{a-1}} \qquad (t\geq 0),
\end{align*}
see Theorem~\ref{thm:scalar_limit_generator}, and agrees with the rescaling limit as $n\to\infty$ of $\{n^{\alpha}\card\Pi(tn^{\beta}), t\geq 0\},$ see Theorem~\ref{thm:coalescent_limit}. However, for $\alpha=-1$ and $\beta=a-1$, we find
\begin{align*}
  c(t) = \left(1+\frac{\Gamma(a+b)}{(2-a)\Gamma(b)}t\right)^{\frac{1}{a-1}} \qquad (t\geq 0).
\end{align*}
If $\Pi$ does not come down from infinity the above rescaling does not admit a diffusion limit. However, we can still study the average behaviour $\{n^{\alpha}\Ex[\card\Pi_n(tn^{\beta})], t\geq 0\}$ of the rescaled number of blocks. Here we find a rescaling limit $m(t)$ for $\alpha=-1,$ $\beta=0$ very different from $c(t),$ namely
\begin{align*}
  m(t) = e^{-\frac{a+b}{a-1}t}\qquad (t\geq 0).
\end{align*}


For beta coalescents that come down from infinity we then study the block size spectrum $(\type{1}\Pi_n(t), \ldots, \type{n}\Pi_n(t))$  that captures more refined information about the coalescent tree. Here $\type{i}\Pi_n(t)$ counts the number of blocks of size $i$ in $\Pi_n(t)$. Using the rescaling $\alpha=-1, \beta=a-1$, the block size spectrum also converges to a deterministic limit as $n\to\infty.$ This limit is characterized by a system of ordinary differential equations whose $i$th solution is a complete Bell polynomial, depending only on $c(t)$ and $a,$ that we work out explicitly, see Corollary~\ref{cor:limit_of_spectrum}.
\end{abstract}
\maketitle

\section{Introduction and summary of results}
Population geneticists are often interested in understanding the genealogy of randomly sampled individuals for a variety of populations. Usually, one starts by describing the evolution of the population forwards in time, classical models for which are the celebrated Wright-Fisher model as well as the Moran model, and numerous variants incorporating more general offspring distributions (the most general framework being the Cannings models), or phenomena such as mutation, selection, age structure, or spatial structure. To  find the genealogy that corresponds to a large population whose evolution is specified forwards in time involves some technical machinery from stochastic processes and was first carried out rigorously by Kingman in~\cite{Kingman1982a, Kingman1982b, Kingman1982c} for a host of population models. Their genealogies turn out to be governed by what is now called Kingman's coalescent. This coalescent, restricted to a sample of $n$ individuals, starts with their $n$ lines of descent. As we trace these ancestral lineages back in time, any pair of lineages merges at rate 1, but no more than two lineages may merge at any given time. For decades after its discovery this stochastic process has been (and still is) utilized by population biologists to model genealogies. However, (a) because of the relevance of non-neutral populations, i.e.~populations with some form of natural selection acting on their individuals, and (b) because of an increasing interest in populations with high fecundity, i.e.~populations where single individuals may beget a number of offspring that is on the order of the total population size, their genealogies have been studied and found to be no longer adequately modeled by Kingman's coalescent. As starting points for more information on this topic the interested reader may consult~\cite{Schweinsberg2015} for developments on populations with selection, and the introduction in~\cite{BirknerBE2013} for developments on populations with high fecundity. The genealogies of samples drawn from these populations turn out to be governed by so-called multiple merger coalescent processes that were introduced independently by Donnelly and Kurtz~\cite{DonnellyKurtz1999}, Pitman~\cite{Pitman1999} and Sagitov~\cite{Sagitov1999}. The multiple merger $n$-coalescent process $\Pi_n=\{\Pi_n(t), t\geq 0\}$ starts with the $n$ lines of descent as does Kingman's coalescent. However, unlike Kingman's coalescent $\Pi_n$ allows for more than two ancestral lines to merge into a single line. In fact, with positive probability all ancestral lines may merge in a single event.

So far we have described the restriction $\Pi_n$ of a multiple merger coalescent to a sample of $n$ individuals. Apparently, we could have restricted ourselves to any sample size $n\geq 2,$ indicating that there might exist an underlying process $\Pi=\{\Pi(t), t\geq 0\}$ governing the mergers of an infinite number of ancestral lines indexed by the natural numbers $\mathbb N\coloneqq \{1, 2, \ldots, \}$, such that the restriction of $\Pi$ to $\{1, \ldots, n\}$ is a process with the same distribution as $\Pi_n$. It turns out that such a process $\Pi$ indeed exists, provided the $(\Pi_n)_{n\geq 2}$ meet some suitable assumption, which seems rather natural from the point of view of sampling. To motivate this assumption, imagine a geneticist collecting a (random) sample of size $n+k$  from a specific population. Unfortunately, on his way to the lab he looses $k$ items in his sample. Clearly, the genealogy of the remaining data is governed by $\Pi_{n+k}$ restricted to $n$ individuals. However, it seems natural to assume that this genealogy should have been the same (in distribution), had the geneticist only collected a sample of size $n$ in the first place. More formally, we assume that for any integers $n\geq 2$ and $k\geq 1$  $\Pi_n$ and the restriction of $\Pi_{n+k}$ to $\{1, \ldots, n\}$ have the same distribution. Under this consistency assumption the projective limit $\Pi$  exists, cf.~\cite{Pitman1999}. Pitman~\cite{Pitman1999} characterized this consistency requirement in terms of finite measures $\Lambda$ on $[0, 1],$ and we will now use this characterization for a formal definition of $\Pi.$ 

\subsection{Multiple merger coalescent processes}
For any finite measure $\Lambda$ on the unit interval there exists a (unique in law) Markov process $\Pi$  with state space $\pset{\mathbb N},$ the set of all set partitions of the positive integers $\mathbb N,$ such that for each $n\in\mathbb N$ the restriction $\Pi_n$ of $\Pi$ to $[n]\coloneqq\{1, \ldots, n\}$ is a continuous-time Markov chain with the following dynamics: when $\Pi_n$ has $m$ blocks, any $2\leq k\leq m$ specific blocks merge into a single block at rate $\lambda_{m, k}\coloneqq \int_0^1 x^{k-2}(1-x)^{m-k}\Lambda(dx)$. The process $\Pi$ is called a multiple merger coalescent process or $\Lambda$-coalescent. To each $\Lambda$ coalescent $\Pi$ one can assign a corresponding tree, the coalescent tree, in an obvious fashion. This is made precise in Section \ref{sec:preliminaries}. We now turn to concrete examples of coalescent processes and hint at some of their connections to stochastic processes in probability or statistical physics.
\begin{example}
(1) Choosing $\Lambda$ to be $\delta_0,$ the Dirac measure that puts mass $1$ on $0,$ yields Kingman's coalescent with transition rates $\lambda_{n,k}=\delta_{k2}.$ This is arguably the most prominent example of a coalescent process. It is the standard model employed by population biologists for the genealogy of a random sample drawn from a large population of haploid individuals. Bertoin and Le Gall~\cite{BertoinLeGall2005} give a construction of the Kingman coalescent via a flow of coalescing diffusions on $[0, 1]$ and another construction via coalescing Brownian motions on the circle. A construction of Kingman's coalescent from a Brownian excursion is given by Berestycki and Berestycki~\cite{BerestyckiBerestycki2009}.

(2) Choosing $\Lambda$ to be the uniform distribution on $[0, 1]$ yields the so-called Bolthausen-Sznitman coalescent with transition rates $\lambda_{n,k}=(k-2)!(n-k)!/(n-1)!.$ This coalescent process was first discovered by Bolthausen and Sznitman~\cite{BolthausenSznitman1998} in the context of the Sherrington-Kirkpatrick model for spin glasses in statistical physics. It also has a natural interpretation, cf.~\cite{BertoinLeGall2000}, as the genealogy of a continuous-state branching process studied by Neveu. A construction of the Bolthausen-Sznitman $n$-coalescent via repeated lifting (or cutting) of a random recursive tree was found by Goldschmidt and Martin~\cite{GoldschmidtMartin2005}.

(3) Choosing $\Lambda$ to be the arcsine distribution with density
\[x\mapsto
  \frac{1}{\pi\sqrt{x(1-x)}}\mathbf 1_{(0, 1)}(x)
\]
yields the so-called arcsine coalescent with transition rates $\lambda_{n,k}=4^{2-n}(k-1)!(n-k+1)!C_{k-2}C_{n-k},$ where $C_n\coloneqq (2n)!/(n!(n+1)!)$ denotes the $n$th Catalan number. In~\cite{Pitters2016a} the author gives a construction of the arscine $n$-coalescent via repeated lifting of a linear preferential attachment tree.

(4) Choosing $\Lambda$ to be the beta$(a, b)$ distribution with density
\[x\mapsto\frac{\Gamma(a+B)}{\Gamma(a)\Gamma(b)}x^{a-1}(1-x)^{b-1}\mathbf 1_{(0, 1)}(x)
\]
yields the beta$(a, b)$ coalescent. Birkner et al.~\cite{BirknerBlathCapaldo2005} found that the genealogy of a continuous-state branching process with $\alpha$ stable branching mechanism is governed by a time-changed beta$(2-\alpha, \alpha)$ coalescent, generalizing the result of Bertoin and Le Gall for Neveu's branching process mentioned earlier. The result of Birkner et al.~also gave rise to an embedding of beta$(2-\alpha, \alpha)$ coalescents ($1<\alpha<2)$ into continuous stable random trees discovered by Berestycki, Berestycki and Schweinsberg~\cite{BerestyckiBerestyckiSchweinsberg2007}.
\end{example}

{\bf Background and previous work.}
In this work we focus on the behaviour of beta$(a, b)$ coalescents $\Pi$ at small times. One idea is that the evolution of $\Pi$ at small times could be approximated by the evolution of $\Pi_n$ as the sample size $n$ grows to infinity. This is reminiscent of what physicists call a hydrodynamic limit, describing the macroscopic evolution of a system comprised of a large number of particles (usually quantified by ordinary or partial differential equations), and deduced from rules dictating the miscroscopic stochastic interactions between individual particles.

{\bf Coagulation processes.}
Hydrodynamic limits have been studied in statistical physics long before geneticists reasoned about coalescents. In fact, the study of coagulation processes goes back at least to the seminal work of Smoluchowski~\cite{Smoluchowski1927} who proposed a class of models for the evolution of a (large) number of particles, where any pair of particles may coagulate to form a new particle. These models are specified by the so-called Smoluchowski's coagulation equations that are parameterized by the coagulation kernel $K(x, y)$ specifying the rate at which a pair of particles with respective masses $x$ and $y$ coagulates. These coagulation models and the models of exchangeable coalescents have precisely one model in common, which obviously is Kingman's coalescent corresponding to a constant coagulation kernel $K(x, y)=1$. So far analytic expressions for hydrodynamic limits are only known for the constant coagulation kernel, the additive kernel $K(x, y)=x+y,$ and the multiplicative kernel $K(x, y)=xy.$ For a survey of stochastic and deterministic models for aggregation and coagulation the reader is referred to Aldous~\cite{Aldous1999}.

{\bf Coalescent processes.}
 Since normalizing the finite measure $\Lambda$ corresponds to a linear time change of $\Pi,$ in what follows we restrict our attention to probability measures, i.e.~$\Lambda([0, 1])=1$. A $\Lambda$ coalescent $\Pi$ is said to come down from infinity if $\Prob{\card\Pi(t)<\infty}=1$ for all $t> 0$ and $\Pi$ is said to stay infinite if $\Prob{\card\Pi(t)=\infty}=1$ for all $t>0$. Pitman~\cite[Proposition 23]{Pitman1999} showed the dichotomy that if $\Lambda$ does not charge $1,$ i.e. $\Lambda(\{1\})=0$, then the corresponding coalescent either comes down from infinity or stays infinite almost surely. Schweinsberg~\cite{Schweinsberg2000} showed that a $\Lambda$ coalescent that does not charge $1$ comes down from infinity if and only if
\begin{align}\label{eq:cdi}
\sum_{n\geq 2}(\gamma_n^{(1)})^{-1}<\infty,
\end{align}
where $\gamma_n^{(1)}\coloneqq \sum_{l=2}^n {n\choose l}\lambda_{n,l}(l-1)$ is the rate at which the number of blocks decreases. We will see that the asymptotic behaviour of 
\[\gamma_n^{(k)}\coloneqq \sum_{l=2}^n \binom{n}{l}\lambda_{n, l}(l-1)^k,\]
as $n\to\infty$ for $k=1, 2, 3$ plays an important r\^{o}le in our analysis of the small time behaviour of $\Pi.$ There is another remarkable characterization of the coming down from infinity of $\Pi$. Consider the Laplace exponent
\[\psi(q)\coloneqq \int_0^1 (e^{-qx}-1+qx)/x^2\Lambda(dx)\qquad (q\geq 0)\]
of a spectrally positive L\'{e}vy process, which is therefore the branching mechanism of a Continuous State Branching Process $X=\{X(t), t\geq 0\},$ say. Bertoin and Le Gall~\cite{BertoinLeGall2006} showed that $\Pi$ comes down from infinity if and only if $X$ becomes extinct in finite time almost surely. According to the so-called Grey's condition, cf.~\cite{Grey1974} and~\cite{Bingham1976}, $X$ becomes extinct in finite time almost surely if and only if
\begin{align}
  \int_1^\infty \frac{dq}{\psi(q)}<\infty.
\end{align}

The study of the small-time behaviour of multiple merger coalescents goes back at least to the work of Berestycki, Berestycki and Limic~\cite{BerestyckiBerestyckiLimic2010}. They show a Law of Large Numbers type of result for the block counting process, namely the almost sure convergence
\begin{align}\label{eq:speed_of_cdi}
  \lim_{t\to 0+} \card\Pi(t)/v_t=1,
\end{align}
where $v$ is uniquely determined by $\int_{v_t}^\infty dq/\psi(q)=t,$ $t>0.$

We study process-valued limits of the rescaled block counting process
\begin{align}\label{eq:rescaled_block_counting_process}
  \{n^\alpha \card\Pi_n(tn^\beta), t\geq 0\},
\end{align}
as well as process-valued limits of the truncated block size spectrum
\begin{align}
  \{n^\alpha (\type{1}\Pi_n(tn^\beta), \ldots, \type{d}\Pi_n(tn^\beta)), t\geq 0\},
\end{align}
as the sample size $n$ grows without bounds, where $\type{i}\pi$ counts the number of blocks of size $i$ in a partition $\pi,$ $\alpha$ and $\beta$ are suitable constants, and $d$ is a fixed positive integer. Interestingly, provided $\Pi$ comes down from infinity, i.e.~if $a<1$, for $\alpha>-1$ and $\beta=\alpha(1-a)$ this limit agrees with the limit of
\begin{align}\label{eq:rescaled_bcp}
  \{n^\alpha \card\Pi(tn^\beta), t\geq 0\}
\end{align}
as $n\to\infty.$ For $\alpha=-1$ and $\beta=a-1$ we still obtain a non-trivial rescaling limit for~\eqref{eq:rescaled_block_counting_process}, however, this limit does not agree with the one for~\eqref{eq:rescaled_bcp}.

 For second-order asymptotics of the block counting process of multiple merger coalescents the reader is referred to~\cite{LimicTalarczyk2015a, LimicTalarczyk2015b, LinMallein2017}. For related work on the small-time behaviour of multiple merger coalescents see~\cite{Sengul2016}.

Before we turn to our main results, we illustrate the basic ideas by discussing a specific example, namely Kingman's coalescent. Strictly speaking, Kingman's coalescent is not a member of the family of beta coalescents (though it can be viewed as a limiting case of the latter), but since it only allows for pairwise mergers, the calculations simplify considerably. Both it's hydrodynamic limit as well as its second-order behaviour have been studied extensively, a nice summary of which can be found in Aldous' survey paper~\cite{Aldous1999}. However, instead of following Aldous' derivation of the hydrodynamic limit we pursue a different approach that can be generalized systematically to multiple merger coalescents.

\subsection{A motivating example: Kingman's coalescent}
The simplest and most celebrated model in population genetics for the genealogy of $n$ chromosomes sampled from a large population of haploid individuals dictates that pairs of lines of descent merge at rate $1$, and no more than two lines of descent may merge at any time. This is the so-called $n$-coalescent discovered by Kingman, cf.~\cite{Kingman1982a},~\cite{Kingman1982b} and~\cite{Kingman1982c}.

Instead of $n$ we may start with a countably infinite number of chromosomes labeled by the positive integers $\mathbb N,$ say. Kingman showed that there exists a Markov process $\Pi\coloneqq\{\Pi(t), t\geq 0\},$ now bearing his name, with initial state $\Pi(0)=\{\{1\}, \{2\}, \ldots \}$, the partition of $\mathbb N$ into singletons, such that for each $n\in\mathbb N$ the restriction $\Pi_n=\{\Pi_n(t), t\geq 0\}$ of $\Pi$ to $\{1, \ldots, n\}$ is an $n$-coalescent. To avoid trivialities, here and in what follows we assume the sample size $n$ to be at least 2. The corresponding genealogical tree has a.s.~finite height
\[\sum_{k\geq 2}\tau_k,\]
since its average height is
\[\Ex \sum_{k\geq 2}\tau_k=\sum_{k\geq 2}\binom{k}{2}^{-1}=2,\]
where $(\tau_k; k\geq 2)$ is a sequence of independent exponentials such that $\tau_k$ has rate $\binom{k}{2}.$ 

\subsubsection*{Block counting process}
Even though we start with an infinite number of particles, there is only a finite number of blocks (or lines of descent) left at any time $t>0$ a.s., a phenomenon dubbed the ``coming down from infinity'' of $\Pi$. To see this, consider the following back-of-the-envelope calculation.
On average, how many blocks do we expect to see at some small time $t>0$?
Since a block is lost at rate $\binom{k}{2}$ whenever there are $k$ blocks in the process, morally, if the average number $\card(t)$ of blocks at time $t$ is sufficiently large, it satisfies the ordinary differential equation
\begin{align}\label{eq:kingman_ode}
  \frac{d}{dt}\card(t) = -\binom{\card(t)}{2}\approx -\frac{1}{2}\card(t)^2,\qquad \card(0)=\infty
\end{align}
with solution $\card(t)=2/t,$ which is finite for all $t>0.$ This heuristics can be found in the proof of Theorem 1 in~\cite{Berestycki2006}. At this point we recall that Kingman's coalescent shares with all other multiple merger coalescents the so-called~\emph{consistency} or \emph{natural coupling} property. This means that if we denote by $\Pi_n=\{\Pi_n(t), t\geq 0\}$ the $n$-coalescent started with $n$ lines of descent, then the processes $\Pi_n$ and $\rho_n\Pi\coloneqq\{\rho_n\Pi(t), t\geq 0\}$ are equal in distribution, where $\rho_n\pi$ denotes the restriction of the partition $\pi$ of $\mathbb N$ to the set $\{1, 2, \ldots, n\}.$
This consistency property immediately translates into the ODEs describing the evolution of the average number $\card_n(t)\coloneqq\Ex[\Pi_n(t)]$ of blocks in $\Pi_n$ which should also solve~\eqref{eq:kingman_ode} but with initial condition $\card_n(0)=n,$ provided that $n$ is large enough. Again, in the limit $n\to\infty$ of unbounded sample size we recover the ODE~\eqref{eq:kingman_ode}. Notice that if instead we rescale time by $n^{-1},$ the relative frequency $c(t)=\card_n(tn^{-1})/n$ of the expected number of blocks in $\Pi_n$ also solves the ODE~\eqref{eq:kingman_ode} with initial condition $c(0)=1$ and is therefore given by $c(t)=2/(2+t),$ independent of $n$.
\newline
In order to better understand these ODEs let us recall a second important property that Kingman's coalescent shares with the multiple merger coalescents, the so-called \emph{temporal coupling}. Namely, let
\begin{align}
  T_n\coloneqq\inf\{t\geq 0\colon \card\Pi(t)=n\}
\end{align}
be the first time at which $\Pi$ reaches a state of $n$ blocks. Then the process $\{\Pi(T_n+t), t\geq 0\}$ started with initial state the partition $\iota\coloneqq\Pi(T_n)$ is equal in distribution to $(\Pi_n|\Pi_n(0)=\iota),$ the $n$-coalescent started in initial state $\iota$ instead of $\Delta_n.$ Informally speaking, if we want to sample an $n$-coalescent, but we only have samples of $\Pi$ at our disposal, we might as well draw a sample from $\Pi$ and start recording its evolution as soon as it jumps into a state of $n$ blocks. In this specific way we can ``trade time for space'', a property sometimes referred to as self-similarity. This observation suggests that we might still obtain the same deterministic limit $c(t)$ if we rescale space by $n^\alpha$ and account for this by rescaling time by $n^\beta$ for some suitably chosen real numbers $\alpha$ and $\beta$. Thus, let us generalize the definition of $c(t)$ to $c_n(t)\coloneqq c_{n, \alpha, \beta}(t)\coloneqq n^\alpha \card_n(tn^\beta),$ and notice that choosing $\alpha=\beta=-1$ we recover the special case that we already considered. According to~\eqref{eq:kingman_ode} $c_n$ solves
\begin{align}
  c_n'(t) &= n^\alpha \frac{d}{dt}\card_n(tn^\beta) = n^\alpha(-\frac{1}{2}\card_n(tn^\beta)^2n^\beta)=-\frac{1}{2}n^{\beta-\alpha}(n^\alpha \card_n(tn^\beta))^2=-\frac{1}{2}n^{\beta-\alpha}c_n(t)^2
\end{align}
with initial condition $c_n(0)=n^{1+\alpha}.$ This shows that if we set $\beta$ equal to $\alpha,$ then as $n\to\infty$ the limit $c$ of $(c_n, n\geq 2)$ solves~\eqref{eq:kingman_ode} with initial condition
\begin{align}
  c(0) &=\begin{cases}\label{eq:kingman_initial_condition}
    0 & \text{if }\alpha<-1,\\
    1 & \text{if }\alpha=-1,\\
    \infty & \text{if }\alpha>-1.
  \end{cases}
\end{align}
In view of~\eqref{eq:kingman_initial_condition} our original case $\alpha=\beta=-1$ (corresponding to a law of large numbers or a hydrodynamic limit) is rather a boundary case. The case $\alpha<-1$ only admits the trivial solution $c(t)=0,$ and we are not going to discuss it further.

The last calculation actually shows more, namely that $c^\star(t)\coloneqq n^\alpha\card(tn^\beta)$ solves
\begin{align}
  \frac{d}{dt}c^\star(t) &= -\frac{1}{2}n^{\beta-\alpha}c^\star(t)^2,\qquad c^\star(0)=\infty,
\end{align}
for any choice of $\alpha$ and $\beta.$ Granted $\alpha=\beta$ this ODE is independent of $n,$ and as before has solution $c^\star(t)=2/t$. Letting $\alpha\coloneqq\beta\coloneqq 0,$ i.e.~there is no rescaling of time nor space, i.e.~we recover $c^\star(t)=\card(t)$.

Looking back, we can now answer the question ``Can we recover the evolution of the average number of blocks in $\Pi$ by studying the limit of the average number of blocks in $\Pi_n$ as the sample size $n$ grows without bounds?'' The answer is positive if in the rescaling $c_n(t)=n^\alpha\card_n(tn^\beta)$ we choose $\alpha=\beta>-1,$ and negative otherwise.

%

\subsubsection*{Block size spectrum}
A natural next step is to ask what the expected number $n_1(t)$ of singletons is at time $t>0$. Since a singleton may merge with either another singleton or a non-singleton, we have to know the number $n(t)-n_1(t)$ of non-singletons and therefore keep track of $n(t)$ in order to extend the previous heuristics. Two singletons are lost whenever a pair of singletons merger, which happens at rate $\binom{n_1(t)}{2}.$ One singleton is lost whenever a singleton merges with another non-singleton block, which happens at rate $n_1(t)(n(t)-n_1(t)).$ Finally, since no  singletons are created at any time in the process, we obtain for $n_1(t)$ the ODE
\begin{align}\label{eq:kingman_ode_singletons}
  n_1(t) &= -2\binom{n_1(t)}{2}-n_1(t)(n(t)-n_1(t))\approx -n(t)n_1(t), \quad (t\geq 0)\qquad n_1(t)=\infty.
\end{align}
As before, after rescaling time by $n^{-1}$ the relative frequency $c_1(t)=n_1(t/n)/n$ of singletons in $\Pi_n$ solves~\eqref{eq:kingman_ode_singletons} with initial condition $c_1(0)=1$ and solution $c_1(t)=c(t)^2$.

For $i>1$ the number $n_i(t)$ of blocks of size $i$ has a more interesting behaviour, since it can also increase whenever two blocks of smaller size merge to form a block of size $i$. More precisely, for any $j,k<i$ such that $j+k=i,$ $n_i(t)$ increases by $1$ whenever a merger of a block of size $j$ with a block of size $k$ occurs at rate $n_j(t)n_k(t).$ Moreover, $n_i(t)$ decreases by $2$ whenever two blocks of size $i$ merge (an event that happens at rate $\binom{n_i(t)}{2}$), and $n_i(t)$ decreases by $1$ whenever a block of size $i$ merges with another block of an unspecified size different from $i$ (at rate $n_i(t)(n(t)-n_{i}(t))$. Overall, this yields a system of coupled ODEs involving $n(t),$ $n_1(t), n_2(t), \ldots, n_i(t).$

One advantage of this approach is that it can be readily and systematically extended from the case of Kingman's coalescent to multiple merger coalescent processes. What remains to be done is to (i) make our heuristic derivation of the coupled system of ODEs rigorous, and (ii) to work out the solution of this ODE system. For part (i), how can we find the limiting ODE system? Conceptually, though the $n$-coalescent does not obey an ODE, the semigroup of its transition probabilities does obey an evolution equation, and this equation is characterized by the corresponding generator $\mathcal G_n,$ say. Consequently, our heuristics suggests that the sequence $(\mathcal G_n)$ has a limit $\mathcal G$ (in an appropriate sense), which in turn encodes the system of coupled ODEs that we are after. This intuition is correct and made formal in Theorem~\ref{thm:scalar_limit_generator} and Proposition~\ref{prop:limit_generator}. The solution of the coupled ODEs is then found in Theorems~\ref{thm:generating_function} and Corollary~\ref{cor:limit_of_spectrum}.

%


\subsection{Main results}
Consider the beta$(a,b)$ coalescent $\Pi$ with parameters $a, b> 0.$ In the first part of this note we study the behaviour of the frequency of the total number of blocks of both $\Pi$ and $\Pi_n$ at small times. To this end, we show in Theorem \ref{thm:scalar_limit_generator} that with a suitable rescaling of time the frequency of the total number of blocks of $\Pi_n$ has a scaling limit, more precisely, as $n\to\infty$ we have convergence to a deterministic limit
\begin{align*}
\{n^{\alpha}\card\Pi_n(tn^{\beta}), t\geq 0\} \to \begin{cases}
      \left(\frac{\Gamma(a+b)}{(2-a)\Gamma(b)}t\right)^{\frac{1}{a-1}} & \text{if }a<1, \alpha>-1, \beta=\alpha(1-a),\\
      \left(1+\frac{\Gamma(a+b)}{(2-a)\Gamma(b)}t\right)^{\frac{1}{a-1}} & \text{if }a<1, \alpha=-1, \beta=a-1
    \end{cases}
\end{align*}
in the Skorokhod topology. We obtain the limit as the solution $c(t)$ of the ordinary differential equation
\begin{align}\label{eq:ode_intro}
  \frac{d}{dt}c(t) = -\frac{\Gamma(a+b)}{(1-a)(2-a)\Gamma(b)}c(t)^{2-a}\quad (t> 0),\qquad   c(0)=\begin{cases}
    1 & \text{if }\alpha=-1,\\
    \infty & \text{if }\alpha\in(-1, 0)
  \end{cases}
\end{align}
of Bernoulli type.

In the second part we restrict ourselves to beta coalescents that come down from infinity, i.e.~$a<1$. We work out in Proposition \ref{prop:limit_generator} a scaling limit (after a suitable rescaling) for the (truncated) block size spectrum
\begin{align}
\{\type{}\Pi_n(t), t\geq 0 \}.
\end{align}
Recall that for any partition $\pi$ of $[n]$ $\type{}\pi\coloneqq (\type{1}\pi, \ldots, \type{n}\pi)$ denotes the so-called type of $\pi$ defined by $\type{i}\pi\coloneqq \card\{B\in\pi\colon \card B=i\}$ for any $i\in\mathbb{N}.$ The block size spectrum can be thought of as a summary statistics that encodes ``almost all'' information on a subtree of the coalescent tree spanned by $n$ leaves. The tree spanned by the leaves $l_1, \ldots, l_n$ is the smallest subtree in the coalescent tree with leaves  $l_1, \ldots, l_n$. To be more specific, given $\{\type{}\Pi_n(t), t\geq 0 \}$ one can recover the corresponding subtree in the coalescent tree up to the choice of branches that merge at each branch point and up to the labelling of the leaves. In fact, in order to reconstruct this subtree on $n$ leaves from the block size spectrum, due to the exchangeability of $\Pi$ all one needs to do is choose at each branch point $k$ branches among the existing branches uniformly at random (if $k$ branches are to merge), and randomly label its leaves by $1, \ldots, n$ (i.e.~according to a permutation of $\{1, \ldots, n\}$ picked uniformly at random). We focus on the behaviour of the average number of blocks of a given size, and therefore rescale the state space of the block size spectrum by $n^{-1}$.

Fix $d\in\mathbb{N}$ arbitrarily. We show that the evolution of the frequency of blocks of size $\leq d$ in $\Pi_n$ with the same time rescaling as before converges to a deterministic limit, namely
\begin{align*}
\{ n^{-1}(\type{1}\Pi_n(tn^{a-1}), \ldots, \type{d}\Pi_n(tn^{a-1})), t\geq 0\}\to \{(c_1(t), \ldots, c_d(t)), t\geq 0\},
\end{align*}
as $n\to\infty$ in $[0, 1]^d$ in the Skorokhod topology, where by Corollary \ref{cor:limit_of_spectrum}
\begin{align}
c_i(t) = \frac{c(t)^{2-a}}{i!}B_i\left(\risfac{\left(\frac{1}{1-a}\right)}{\bullet}(-c(t)^{1-a})^{\bullet-1}, \risfac{(1-a)}{\bullet}\right)\quad (i\in\mathbb{N}).
\end{align}
Here, for any two sequences $v_\bullet=(v_k)_{k\in\mathbb{N}}$ and $w_\bullet=(w_k)_{k\in\mathbb N}$
\[ B_i(v_\bullet, w_\bullet)\coloneqq \sum_{l=1}^iv_l B_{i,l}(w_\bullet)\]
denotes the $i$th complete Bell polynomial (associated with $(v_\bullet, w_\bullet)$), where
\[ B_{i,l}(w_\bullet)\coloneqq \sum_{\pi\in\pset{[i],l}}\prod_{B\in\pi}w_{\card B}\]
denotes the $(i,l)$th partial Bell polynomial (associated with $w_\bullet$) and $\pset{[i],l}$ denotes the set of all partitions of $[i]$ that contain $l$ blocks. Moreover, for $x\in\mathbb{R},$ $k\in\mathbb N$ let $\fallfac{x}{k}\coloneqq x(x-1)\cdots (x-k+1)$ denote the falling factorial power, $\risfac{x}{k}\coloneqq x(x+1)\cdots (x+k-1)$ the rising factorial power and we agree on $\fallfac{x}{0}\coloneqq\risfac{x}{0}\coloneqq 1,$ and for any function $f\colon\mathbb N\to\mathbb R$ we write $f(\bullet)$ as a shorthand for the sequence $(f(k))_{k\geq 1}$.


We find the functions $c_i(t)$ by computing their generating function
\begin{align*}
\mathscr G(t, x)\coloneqq \sum_{i\geq 1}c_i(t)x^i\quad (t\geq 0, x\in [0, 1]).
\end{align*}
Theorem \ref{thm:generating_function} states that $\mathscr G$ is given by
\begin{align}\label{eq:intro_generating_function}
\mathscr{G}(t, x) &= c(t)-\left((1-x)^{a-1}+\frac{\Gamma(a+b)}{(2-a)\Gamma(b)}t\right)^{\frac{1}{a-1}}\quad (x\in (-1, 1), t\geq 0).
\end{align}
It is remarkable to see the similarity between the subtrahend in formula \eqref{eq:intro_generating_function} for $\mathscr G$, namely
\begin{align*}
g(t, x)\coloneqq c(t)-\mathscr{G}(t, x)=\left((1-x)^{a-1}+\frac{\Gamma(a+b)}{(2-a)\Gamma(b)}t\right)^{\frac{1}{a-1}}
\end{align*}
and $c(t)$. In fact, $g(t, x)$ solves the partial differential equation
\begin{align}
\partial_t g(t, x) = -\frac{\Gamma(a+b)}{(1-a)(2-a)\Gamma(b)}g(t, x)^{2-a},
\end{align}
with boundary condition $g(0, x)=1-x$. This partial differential equation should be compared to the ordinary differential equation for $c(t),$ equation \eqref{eq:ode_intro}. We interpret this as a form of self-similarity of the limiting block frequency spectrum of the beta coalescents expressed in terms of generating functions.

\section{Preliminaries}\label{sec:preliminaries}
A partition of a set $A$ is a set, $\pi$ say, of nonempty pairwise disjoint subsets of $A$ whose union is $A$. The members of $\pi$ are called the blocks of $\pi$. Let $\card A$ denote the cardinality of $A$ and let $\pset{A}$ denote the set of all partitions of $A$.

Let $(\Omega, \mathcal F, \mathbb P)$ denote the probability space underlying $\Pi$. If $\Pi$ comes down from infinity, we will (as is often done implicitly) identify for any $\omega\in\Omega$ the path $t\mapsto\Pi(t)(\omega)$ with a rooted tree whose leaves are labelled by $\mathbb{N}.$ More formally, the set of nodes of the tree corresponding to $t\mapsto \Pi(t)(\omega)$ is
\begin{align*}
\mathcal T(\omega)\coloneqq \{(t, B)\colon t\geq 0, B\in\Pi(t)(\omega)\}.
\end{align*}
If we interpret $\mathcal T(\omega)$ as a genealogical tree, $(s, B)\in\mathcal T(\omega)$ means that individual $B$ is alive at time $s,$ and if for two points $(s, B), (t, C)\in\mathcal T(\omega)$ we have that $s\leq t$ and $B\subsetneq C,$ then $C$ is interpreted as an ancestor of $B$ alive at time $t$.  For any two points $(s, B), (t, C)\in\mathcal T(\omega)$  let
\begin{align*}
m((s, B), (t, C))(\omega)\coloneqq \inf\{u>s\vee t\colon \text{both $A$ and $B$ are subsets of a common block in }\Pi(u)(\omega)\}
\end{align*}
denote the time back to the most recent common ancestor of $(s, B)$ and $(t, C).$ It can be shown that $\mathcal T(\omega)$ together with the metric $d(\omega)$ defined by
\begin{align*}
d((s, B), (t, C))(\omega)\coloneqq (m((s, B), (t, C))(\omega)-s) + (m((s, B), (t, C))(\omega)-t)
\end{align*}
is an $\mathbb R$-tree. This is done formally in Example 3.41 of \cite{Evans2005}. Informally, $d(\omega)$ yields the genealogical distance between any two points in $\mathcal T(\omega)$.

One may wonder whether the (random) tree $\mathcal T$ can be described more explicitly. One way to study $\mathcal T$ is by ``exploring'' it via subtrees, namely, if we consider any $n$ of its leaves labelled $l_1,\ldots, l_n\in\mathbb N$, their spanning tree will correspond to a $\Lambda$ $n$-coalescent with leaves labelled $l_1, \ldots, l_n,$ as is apparent from the consistency of $\Lambda$ coalescents. As we increase the sample size $n,$ we explore larger and larger subtrees of $\mathcal T$. However, the topology of a subtree spanned by $n$ leaves is rather involved.


In order to work out explicitly the asymptotic behaviour of a subtree when the number $n$ of leaves grows without bound in what follows, we restrict ourselves to beta coalescents that come down from infinity, i.e.~we take $\Lambda$ to be the beta$(a, b)$ distribution with density
\begin{align}
\Lambda(dx) =  B(a, b)^{-1}x^{a-1}(1-x)^{b-1}dx\quad (x\in (0, 1)),
\end{align}
where for $a, b>0$ the beta function with parameters $a, b$ is defined as $B(a, b) \coloneqq\int_0^1 x^{a-1}(1-x)^{b-1}dx.$ Notice that according to Schweinsberg's characterization of coalescents that come down from infinity, equation \eqref{eq:cdi}, the beta$(a,b)$ coalescent comes down from infinity if and only if $a\in (0, 1),$ cf.~Example 15 in \cite{Schweinsberg2000}. Without further mention, we assume hereafter that $\Lambda$ is the beta$(a, b)$ distribution for some $a\in (0, 1),$ $b>0.$

We use the convention that empty sums equal $0$ and empty products equal $1$ throughout.

\section{Block counting process}
Before we turn to the scaling limit of the block size spectrum we study the simpler block counting process $N_n\coloneqq \{N_n(t), t\geq 0\}.$ Recall that $N_n(t)=\card\Pi_n(t)$ counts the number of blocks in $\Pi_n(t)$. Moreover, recall that the Gamma function is defined as $\Gamma(a)\coloneqq\int_0^1 e^{-x}x^{a-1}dx$ for any positive real number $a\in (0, \infty)$ and may be meromorphically continued to the entire complex plane. We will repeatedly use the identity $B(a, b) = \Gamma(a)\Gamma(b)/\Gamma(a+b)$ for $a, b>0$.

In order to find the correct time scaling for $N_n$, notice that the rate at which the number of blocks decreases has asymptotics
\begin{align*}
\sum_{l=2}^n {n\choose l}\lambda_{n,l}(l-1) \sim \frac{\Gamma(a+b)}{(1-a)(2-a)\Gamma(b)}n^{2-a},
\end{align*}
as $n\to\infty,$ which is proved in Lemma \ref{lem:order_of_gamma1}. Consequently, in order to see a nontrivial limit of $N_n$ when its state space is rescaled by $n^\alpha,$ we should rescale time by a factor on the order of $n^{(1-a)\alpha}.$ Let therefore $C_n\coloneqq \{C_n(t), t\geq 0\}$ be defined by
\begin{align}
C_n(t)\coloneqq n^\alpha N_n(t\tau_n)\qquad (n\geq 2, t\geq 0),
\end{align}
where $\tau_n\sim n^{(1-a)\alpha}$ as $n\to\infty.$

The process $C_n$ is a continuous-time Markov chain with state space $E_n\coloneqq n^\alpha[n],$ initial state $C_n(0)=n^{1+\alpha},$ absorbing state $n^\alpha$ and evolves according to the following dynamics:
\begin{align}
 \text{a transition }\qquad c\mapsto c-n^\alpha(l-1)\qquad\text{ occurs at rate }\quad\binom{n^{-\alpha}c}{l}\lambda_{n^{-\alpha}c, l}\qquad(2\leq l\leq n^{-\alpha}c).
\end{align}
For $2\leq k\leq m$ if there are currently $m$ blocks in $\Pi$, we will see any $k$ specific blocks merge at rate
\begin{align}\label{eq:beta_coalescent_rates}
  \lambda_{m,k}&\coloneqq \int_0^1 x^{k-2}(1-x)^{m-k}\Lambda(dx) = \frac{B(k-2+a, m-k+b)}{B(a, b)}=\frac{\risfac{a}{k-2}\risfac{b}{m-k}}{\risfac{(a+b)}{m-2}}.
\end{align}

\begin{remark}
Equation \eqref{eq:beta_coalescent_rates} yields the recursive formula
\begin{align}
  \lambda_{m, k+1} =\frac{a+k-2}{b+m-k-1}\lambda_{m,k} \qquad (2\leq k\leq m-1).
\end{align}
This should be compared to the recursive formula
\[\lambda_{m,k}=\lambda_{m+1,k}+\lambda_{m+1, k+1} \quad (2\leq k\leq m)\]
for arbitray $\Lambda$ given by Pitman in~\cite{Pitman1999}, Lemma 18. Combined, these formulae yield
\begin{align*}
\lambda_{m+1,k} &= \frac{b+m-k}{a+b+m-2}\lambda_{m,k}\quad (2\leq k\leq m),
\end{align*}
and can be used to efficiently compute the rates of the beta$(a,b)$ coalescent by an algorithm.
\end{remark}

\subsection{Average behaviour}
Before we study the limit of the rescaled block counting process $C_n(t)=n^\alpha N_n(t\tau_n)$ as the sample size $n$ grows without bounds it is instructive to work out the limit for its average $m_n(t)\coloneqq\Ex C_n(t).$ Notice that $m_n(t)$ in fact depends on the parameters $a, b, \alpha$ and the sequence $\tau_n$ (which in turn depends on $\beta$), but in order not to overburden notation we will suppress this dependence.

Let us define
\begin{align}
  m_n(t, h)\coloneqq \Ex[C_n(t+h)|C_n(t)=c]\qquad (t\geq 0, c\in E_n),
\end{align}
and notice that this quantity does not depend on $t,$ since $C_n$ is a time-homogeneous Markov chain.

\begin{proposition}\label{prop:asymptotics}
  For any $t\geq 0$ we have for fixed $n$ and $c\in E_n$
  \begin{align}
    m_n(t, h) &= -n^{\alpha+\beta}\gamma_{n^{-\alpha}c}^{(1)}h+c+o(h).
  \end{align}
\end{proposition}
\begin{proof}
Recall that from the infinitesimal characterization of the transition probabilities of a continuous-time Markov chain on a finite state space, cf.~\cite[Theorem 2.8.2]{Norris1998}, we have
  \begin{align}
    \Prob{N_n(t+h)=b-l+1|N_n(t)=b} &= \begin{cases}
      \binom{b}{l}\lambda_{b, l}h+o(h) & \text{if }2\leq l\leq b\\
      1-\lambda_bh+o(h) & \text{if }l=2,
    \end{cases}
  \end{align}
as $h\downarrow 0.$ We therefore compute
\begin{align*}
  m_n(0, h) &= \Ex[C_n(h)|C_n(0)=c]\\
  &= \sum_{l=1}^{n^{-\alpha} c}(c-n^{\alpha}(l-1)) \Prob{C_n(h)=c-n^{\alpha}(l-1)|C_n(0)=c}\\
  &= c\left(\sum_{l=2}^{n^{-\alpha}c}\binom{n^{-\alpha}c}{l}\lambda_{n^{-\alpha}c, l}hn^\beta+o(h)\right)-n^\alpha\sum_{l=2}^{n^{-\alpha}c}\binom{n^{-\alpha}c}{l}\lambda_{n^{-\alpha}c, l}(l-1)hn^\beta\\
  &\quad-n^\alpha\sum_{l=2}^{n^{-\alpha}c}(l-1)o(h)+c(1-\lambda_{n^{-\alpha}c}hn^\beta+o(h))\\
  &= -n^{\alpha+\beta}\gamma_{n^{-\alpha}c}^{(1)}h+c+o(h).
\end{align*}
\end{proof}

From this proposition we can derive a differential equation that governs $m_n(t).$
\begin{corollary}\label{cor:mean_ode}
For any $n\geq 2$ $m_n(t)$ solves the differential equation
\begin{align}
  \frac{d}{dt}m_n(t) &= -n^{\alpha+\beta}\gamma_{n^{-\alpha}m_n(t)}^{(1)}\quad (t>0),\qquad m_n(0) = n^{1+\alpha}.
\end{align}
\end{corollary}
\begin{proof}
Let $M_n(t, h)\coloneqq\Ex[C_n(t+h)|C_n(t)].$ Using Proposition~\ref{prop:asymptotics} we have
\begin{align*}
  m_n(t, h) &= \Ex C_n(t+h)=\Ex[\Ex[C_n(t+h)|C_n(t)]]=\Ex[M_n(t, h)]=\Ex[M_n(0, h)],
  \intertext{so}
  \frac{m_n(t+h)-m_n(t)}{h} &=\frac{\Ex C_n(t+h)-\Ex C_n(t)}{h}\\
  &= \frac{\Ex[\Ex[C_n(t+h)|C_n(t)]]-\Ex[\Ex[C_n(t)|C_n(t)]]}{h}\\
  &= \frac{\Ex[M_n(t+h)-M_n(t, 0)]}{h}=\frac{-n^{\alpha+\beta}\gamma_{n^{-\alpha}c}h+o(h)}{h}\to -n^{\alpha+\beta}\gamma_{n^{-\alpha}c}^{(1)}
\end{align*}
as $h\downarrow 0,$
and $m_n(0)=C_n(0)=n^{\alpha}N_n(0)=n^{1+\alpha}.$
\end{proof}

We prepare the proof of the scaling limit for $m_n$ by establishing some Lemmata. For $k\in\mathbb{N}$ define
\begin{align}
  \gamma_n^{(k)}\coloneqq \sum_{l=2}^n \binom{n}{l}\lambda_{n, l}(l-1)^k,\quad\text{ and }\quad\gamma^{(\underline{k})}_n\coloneqq \sum_{l=2}^n {n\choose l}\lambda_{n,l}\fallfac{l}{k}.
\end{align}

The main actor in the proof of our first result about the rescaling limit of $m_n(t)$, Theorem~\ref{thm:ode_mean}, is $\gamma_n^{(1)},$ as Corollary~\ref{cor:mean_ode} suggests. In particular, we are interested in the asymptotic behavior of $\gamma_n^{(1)}$ that we establish in the following Lemmata.
Later, when we study the rescaling limit of $C_n,$ the fluctuations about its mean $m_n$ and the asymptotic behavior of $\gamma_n^{(2)}$ and $\gamma_n^{(3)}$ will be of importance.
\begin{lemma}\label{lem:risfac_ratio_asymptotics}
For $a, b>0,$ a natural number $n\in\mathbb{N}$ and an integer $z\in\mathbb{Z}$ we have
\begin{align}
  \frac{\risfac{a}{n}}{\risfac{b}{n+z}} \sim \frac{\Gamma(b)}{\Gamma(a)}n^{a-b-z},
\end{align}
as $n\to\infty$.
\end{lemma}
\begin{proof}
We calculate
\begin{align*}
 \frac{\risfac{a}{n}}{\risfac{b}{n+z}} &= \frac{\Gamma(a+n)}{\Gamma(a)}\frac{\Gamma(b)}{\Gamma(b+n+z)}\sim\frac{\Gamma(b)}{\Gamma(a)}n^{a-b-z},
\end{align*}
as $n\to\infty$.
\end{proof}
We now need some notation. For $d\in\mathbb N,$ $x\in\mathbb{R}^d$ and $k\in\mathbb{N}_0^d$ let $\risfac{x}{k}\coloneqq \prod_{i=1}^d \risfac{x_i}{k_i}$ and $\abs x=\sum_{i=1}^d \abs{x_i}.$
\begin{lemma}\label{lem:super-lemma}
Fix $d\in\mathbb N$ and $k, n\in\mathbb{N}_0^d$. Then for $a, b\in\mathbb R$
\begin{align}\label{eq:super-lemma}
\sum_{l\in\mathbb{N}_0^d} \fallfac{l}{k}\risfac{a}{\abs l}\risfac{b}{n-\abs l}\prod_{i=1}^d {n_i\choose l_i} &= \risfac{a}{\abs k}\fallfac{n}{\abs k}\risfac{(a+b+\abs k)}{\abs{n}-\abs{k}}.
\end{align}
\end{lemma}
\begin{proof}
We first prove the statement for $k=0$ by an induction on $d$. Hence, for $d=1$ the statement reads
\begin{align*}
\sum_{l=0}^n{n\choose l}\risfac{a}{l}\risfac{b}{n-l} &= \risfac{(a+b)}{n},
\end{align*}
and this is true, since the sequence $(\risfac{a}{k})_{k\geq 1}$ of rising factorial powers is a sequence of polynomials of binomial type, as is well known. Suppose now that \eqref{eq:super-lemma} holds for some $d\in\mathbb{N}$. Then
\begin{align*}
&\sum_{l\in\mathbb{N}_0^{d+1}} \risfac{a}{\abs l}\risfac{b}{\abs n-\abs l}\prod_{i=1}^{d+1} {n_i\choose l_i}\\
 &= \sum_{l\in\mathbb{N}_0^d} \risfac{a}{\abs l}\risfac{b}{n_1+\ldots+n_d-\abs l}\prod_{i=1}^d{n_i\choose l_i} \sum_{l_{d+1}=0}^{n_{d+1}}{n_{d+1}\choose l_{d+1}}\risfac{(a+\abs l)}{l_{d+1}}\risfac{(b+n_1+\cdots+n_d-\abs l)}{n_{d+1}-l_{d+1}}\\
 &= \risfac{(a+b)}{n_1+\cdots+n_d}\risfac{(a+b+n_1+\ldots+n_d)}{n_{d+1}}=\risfac{(a+b)}{\abs n},
\end{align*}
where we used the induction hypothesis in the second equality. Now suppose $\abs k>0.$ Performing the index shift $m=l-k$ in the first step and applying the statement for $k=0$ in the last step, we obtain
\begin{align*}
\sum_{l\in\mathbb{N}_0^d}\fallfac{l}{k}\risfac{a}{\abs l}\risfac{b}{\abs n-\abs l}\prod_{i=1}^d {n_i\choose l_i} &= \sum_{m\in\mathbb{N}_0^d}\fallfac{(m+k)}{k}\risfac{a}{\abs m+\abs k}\risfac{b}{\abs n-\abs k-\abs m}\prod_{i=1}^d {n_i\choose m_i+k_i}\\
&= \risfac{a}{\abs k}\sum_{m\in\mathbb{N}_0^d}\risfac{(a+\abs k)}{\abs m}\risfac{b}{\abs n-\abs k-\abs m}\prod_{i=1}^d \left[\fallfac{(m_i+k_i)}{k_i}{n_i\choose m_i+k_i}\right]\\
&= \risfac{a}{\abs k}\fallfac{n}{k}\sum_{m\in\mathbb{N}_0^d}\risfac{(a+\abs k)}{\abs m}\risfac{b}{\abs n-\abs k-\abs m}\prod_{i=1}^d {n_i-k_i\choose m_i}\\
&=\risfac{a}{\abs k}\fallfac{n}{k}\risfac{(a+b+\abs k)}{\abs n-\abs k},
\end{align*}
which completes the proof.
\end{proof}

\begin{lemma}\label{lem:order_of_gamma0}
For $a\notin \{1, 2\}$ as $n\to\infty$ one has 
  \begin{align}
    \gamma_n^{(0)}\sim\begin{cases}
      \frac{\Gamma(a+b)}{(2-a)\Gamma(b)}n^{2-a} & \text{if }a<2\\
      \frac{(a+b)(a+b-1)}{(a-1)(a-2)} & \text{if }a>2.
    \end{cases}
  \end{align}
\end{lemma}
\begin{proof}
We have
\begin{align*}
\gamma_n^{(0)} = \sum_{l=2}^n{n\choose l}\lambda_{n,l} &= \frac{1}{\risfac{(a+b)}{n-2}}\sum_{l=2}^n{n\choose l}\risfac{a}{l-2}\risfac{b}{n-l}\\
&= \frac{1}{(1-a)(2-a)\risfac{(a+b)}{n-2}}\sum_{l=2}^n{n\choose l}\risfac{(a-2)}{l}\risfac{b}{n-l}\\
&= \frac{1}{(1-a)(2-a)}\left(\frac{\risfac{(a+b-2)}{n}}{\risfac{(a+b)}{n-2}}+\frac{(2-a)n\risfac{b}{n-1}-\risfac{b}{n}}{\risfac{(a+b)}{n-2}}\right)\\
&\sim \frac{1}{(1-a)(2-a)}\left(\frac{\Gamma(a+b)}{\Gamma(a+b-2)}+\frac{(2-a)\Gamma(a+b)}{\Gamma(b)}n^{2-a}-\frac{\Gamma(a+b)}{\Gamma(b)}n^{2-a}\right)\\
&\sim\begin{cases}
 \frac{\Gamma(a+b)}{(2-a)\Gamma(b)}n^{2-a} & \text{if } a<2\\
 \frac{(a+b)(a+b-1)}{(a-1)(a-2)} & \text{if }a>2,
 \end{cases}
\end{align*}
where we applied Lemmata~\ref{lem:risfac_ratio_asymptotics} and~\ref{lem:super-lemma} and distinguished the cases $a+b=2$ and $a+b\neq 2$.
\end{proof}

\begin{lemma}\label{lem:order_of_gamma1}
For $a\notin \{1, 2\}$ as $n\to\infty$ we have that
\begin{align}
\gamma_n^{(1)}
\sim\begin{cases}
  \frac{\Gamma(a+b)}{(1-a)(2-a)\Gamma(b)}n^{2-a} & \text{if }a<1\\
  \frac{a+b}{a-1}n & \text{if }a>1.
  \end{cases}
\end{align}
\end{lemma}
\begin{proof}
Firstly, note the relation $\gamma_n^{(1)}=\gamma_n^{(\underline 1)}-\gamma_n^{(0)}$. Using Lemma~\ref{lem:super-lemma} with $d=k=1,$ we find
\begin{align*}
\gamma_n^{(\underline 1)}=-\frac{n}{1-a}\left(\frac{\risfac{(a+b-1)}{n-1}}{\risfac{(a+b)}{n-2}}-\frac{\risfac{b}{n-1}}{\risfac{(a+b)}{n-2}}\right).
\end{align*}
The first summand in above paranthesis vanishes if $a+b=1.$ If $a+b\neq 1$ we have
\begin{align*}
\frac{\risfac{(a+b-1)}{n-1}}{\risfac{(a+b)}{n-2}}\sim \frac{\Gamma(a+b)}{\Gamma(a+b-1)}
\end{align*}
as $n\to\infty$ by Lemma \ref{lem:risfac_ratio_asymptotics}. On the other hand, Lemma \ref{lem:risfac_ratio_asymptotics} yields
\begin{align*}
\frac{\risfac{b}{n-1}}{\risfac{(a+b)}{n-2}}\sim \frac{\Gamma(a+b)}{\Gamma(b)}n^{1-a},
\end{align*}
as $n\to\infty,$ and we obtain
\begin{align}
\gamma_n^{(\underline 1)}\sim\begin{cases}
  \frac{\Gamma(a+b)}{(1-a)\Gamma(b)}n^{2-a} & \text{if }a<1,\\
  \frac{a+b}{a-1}n & \text{if }a>1,
  \end{cases}
\end{align}
as $n\to\infty.$ Putting everything together and applying Lemma~\ref{lem:order_of_gamma0} the claim follows.
\end{proof}

\begin{theorem}\label{thm:ode_mean}
Suppose $a\notin\{1, 2\}$. As $n\to\infty$ the rescaled average number of blocks $m_n(t)$ converges to the solution $m(t)$ of the ODE
\begin{align}
  \frac{d}{dt}m(t) &= \begin{cases}\label{eq:ode_mean}
    -\frac{\Gamma(a+b)}{(1-a)(2-a)\Gamma(b)}m(t)^{2-a} &\text{if }a<1, \beta=\alpha(1-a),\\
    -\frac{a+b}{a-1}m(t)&\text{if }a>1, \beta=0,
  \end{cases}
  \qquad m(0)=\begin{cases}
    0 & \text{if }\alpha<-1,\\
    1 & \text{if }\alpha=-1,\\
    \infty & \text{if }\alpha>-1.
  \end{cases}
\end{align}
The non-trivial solutions of these ODEs are given by
\begin{align*}
  m(t) &= \begin{cases}
    \left(\frac{\Gamma(a+b)}{(2-a)\Gamma(b)}t\right)^{\frac{1}{a-1}}, & \text{if }a<1, \alpha> -1, \beta=\alpha(1-a),\\
    \left(1+\frac{\Gamma(a+b)}{(2-a)\Gamma(b)}t\right)^{\frac{1}{a-1}}, & \text{if }a<1, \alpha= -1, \beta=a-1,\\
    e^{-\frac{a+b}{a-1}t} & \text{if }a>1,\alpha=-1, \beta=0,\\
  \end{cases}
\end{align*}
for any $t\geq 0.$
\end{theorem}
\begin{proof}
xxx Argue that convergence of differential equations to the ODE implies convergence of their solutions. From Corollary~\ref{cor:mean_ode} and applying Lemma~\ref{lem:order_of_gamma1} we have that the  differential equations for $m_n$ converge as $n\to\infty$ to
\begin{align*}
  m'(t) = -n^{\alpha+\beta}\gamma_{n^{-\alpha}m(t)}^{(1)}&\sim \begin{cases}
  -\frac{\Gamma(a+b)}{(1-a)(2-a)\Gamma(b)}n^{\alpha+\beta}n^{-\alpha(2-a)}m(t)^{2-a} & \text{if }a<1,\\
  -\frac{a+b}{a-1}n^{\beta}m(t) & \text{if }a>1,
  \end{cases}\\
  &= \begin{cases}
  -\frac{\Gamma(a+b)}{(1-a)(2-a)\Gamma(b)}n^{\alpha(a-1)+\beta}m(t)^{2-a} & \text{if }a<1,\\
  -\frac{a+b}{a-1}n^{\beta}m(t) & \text{if }a>1.
  \end{cases}\\
\end{align*}
If $\Pi$ comes down from infinity, the last display shows that the derivative $m'(t)$ is neither zero nor unbounded if and only if $\alpha(a-1)+\beta=0$. On the other hand, if $\Pi$ does not come down from infinity, we have to require $\beta=0$ in order for $m'(t)$ not to vanish or be unbounded. For the boundary condition, Corollary~\ref{cor:mean_ode} yields
\begin{align*}
  m_n(0) &= n^{1+\alpha}\to m(0)\coloneqq\begin{cases}
    0 & \text{if }\alpha< -1,\\
    1 & \text{if }\alpha=-1,\\
    \infty & \text{if }\alpha>-1,
  \end{cases}
\end{align*}
as $n\to\infty.$

Solving these ODEs for $\alpha=-1$ is an exercise in ordinary differential equations, and we leave the details to the reader. If $\Pi$ does not come down from infinity, the cases $\alpha<-1$ and $\alpha>-1$ only have trivial solutions. If $\Pi$ comes down from infinity, we obtain the solution for $\alpha>-1$ as follows. For any $M>0$ let $m_{M}$ denote the solution of~\eqref{eq:ode_mean} with initial condition $m_{M}(0)=M,$ i.e.~$m_{M}(t)=\left(M^{a-1}+\frac{\Gamma(a+b)}{(2-a)\Gamma(b)}t\right)^{\frac{1}{a-1}}.$ By definition $\lim_{M\to\infty} m_{M}(0)=m(0),$ and, trivially, the $m'_{M}$ converge to $m'$ uniformly as $M\to\infty,$ since they are all identical. This implies the existence of a function $m$ such that the $m_{M}$ converge uniformly to $m$ and $m'(t)=\lim_{M\to\infty}m'_{M}(t)=-\frac{\Gamma(a+b)}{(1-a)(2-a)\Gamma(b)}m(t)^{2-a}.$ Thus $m(t)=\left(\frac{\Gamma(a+b)}{(2-a)\Gamma(b)}t\right)^{\frac{1}{a-1}}$ solves~\eqref{eq:ode_mean} with initial condition $m(0)=\infty$ as required.
\end{proof}

\subsection{Process-valued rescaling limits}



\begin{lemma}\label{lem:order_of_gamma2}
  For $a\notin\{1, 2\}$ we have $\gamma_n^{(2)}, \gamma_n^{(\underline 2)}\sim n^2$ as $n\to\infty.$
\end{lemma}
\begin{proof}
We can write $\gamma_n^{(2)} = \gamma_n^{(\underline{2})}-\gamma_n^{(1)},$ since $(l-1)^2=\fallfac{l}{2}-(l-1).$ Applying Lemma \ref{lem:super-lemma} with $d=1$ and $k=2$ we find that
\begin{align*}
\gamma_n^{(\underline{2})} &= \sum_{l=2}^n {n\choose l}\frac{\risfac{a}{l-2}\risfac{b}{n-l}}{\risfac{(a+b)}{n-2}}\fallfac{l}{2}= \frac{1}{\risfac{(a-2)}{2}\risfac{(a+b)}{n-2}}\sum_{l=2}^n{n\choose l}\risfac{(a-2)}{l}\risfac{b}{n-l}\fallfac{l}{2}=n(n-1).
\end{align*}
From this and Lemma \ref{lem:order_of_gamma1} we conclude $\gamma^{(2)}_n=\gamma_n^{(\underline{2})}-\gamma_n^{(1)}\sim n^2$ as $n\to\infty.$
\end{proof}

\begin{lemma}\label{lem:order_of_gamma3}
  For $a\notin\{1, 2\}$ we have $\gamma_n^{(3)}\sim\frac{a}{a+b+1}n^3.$
\end{lemma}
\begin{proof}
  Notice that $(l-1)^3=\fallfac{l}{3}+(l-1)^2-\fallfac{(l-1)}{2},$ hence $\gamma_n^{(3)}=\gamma_n^{(\underline{3})}+\gamma_n^{(2)}-\gamma_n^{(\underline{2})}.$ From Lemma~\ref{lem:order_of_gamma2} we have $\gamma_n^{(2)}-\gamma_n^{(\underline{2})}\sim 0$ as $n\to\infty.$ Now,
  \begin{align*}
    \gamma_n^{(\underline 3)} &= \sum_{l=2}^n \binom{n}{l}\lambda_{n, l}\fallfac{l}{3} = \sum_{l=2}^n\binom{n}{l}\frac{\risfac{a}{l-2}\risfac{b}{n-l}}{\risfac{(a+b)}{n-2}}\fallfac{l}{3}\\
    &= \frac{1}{(a-2)(a-1)\risfac{(a+b)}{n-2}}\sum_{l=2}^n \binom{n}{l}\risfac{(a-2)}{l}\risfac{b}{n-l}\fallfac{l}{3}\\
    &= \frac{1}{(a-2)(a-1)\risfac{(a+b)}{n-2}}\risfac{(a-2)}{3}\fallfac{n}{3}\risfac{(a+b+1)}{n-3}\\
    &\sim\frac{a}{a+b+1}n^3
  \end{align*}
as $n\to\infty.$
\end{proof}

For a metric space $(E, r)$ we denote by $D_E([0, \infty))$ the space of right-continuous functions from $[0, \infty)$ into $E$ having left limits. Moreover, by $C(E),$ respectively $C^\infty(E),$ we denote the continuous, respectively smooth functions (that is functions that have derivatives of arbitrary order) from $E$ to $\mathbb R$.
\begin{theorem}\label{thm:scalar_limit_generator}
Fix $\alpha\in [-1, 0)$ and let $\tau_n$ be of order $n^{(1-a)\alpha}$. Then as $n\to\infty$ we have convergence
  \begin{align}
  \{C_n(t), t\geq 0\}\to \{c(t), t\geq 0\}
  \end{align}
in $D_{[0, 1]}([0, \infty))$ in the Skorokhod topology, where $c(t)$ solves the ordinary differential equation of Bernoulli type
\begin{align}\label{eq:ode}
  \frac{d}{dt}c(t) = -\frac{\Gamma(a+b)}{(1-a)(2-a)\Gamma(b)}c(t)^{2-a}\quad (t\geq 0),
\end{align}
  with boundary condition
\begin{align}
  c(0)=\begin{cases}
    1 & \text{if }\alpha=-1,\\
    \infty & \text{if }\alpha\in(-1, 0).
  \end{cases}
\end{align}
The solution of~\eqref{eq:ode} is given by
  \begin{align}\label{eq:ode_solution}
    c(t) &= \begin{cases}
      \left(\frac{\Gamma(a+b)}{(2-a)\Gamma(b)}t\right)^{\frac{1}{a-1}} & \text{if }a<1, \alpha>-1, \beta=\alpha(1-a),\\
      \left(1+\frac{\Gamma(a+b)}{(2-a)\Gamma(b)}t\right)^{\frac{1}{a-1}} & \text{if }a<1, \alpha=-1, \beta=a-1.
    \end{cases}
  \end{align}
  \end{theorem}
  
\begin{remark}
We omit the case $\alpha<-1$ as it corresponds to the initial condition $c(0)=0$ and trivial solution $c(t)=0,$ $t\geq 0$.
\end{remark}

\begin{remark}
(1)~The special case $a=\frac{1}{2}$ is interesting as it contains the arcsine coalescent which was recently studied in~\cite{Pitters2016a}. Using Legendre's duplication formula $\Gamma(2z)=2^{2z-1}\Gamma(z)\Gamma(z+\frac{1}{2})/\sqrt{\pi},$ we find
  \begin{align}
    c(t) = \left(\frac{3\Gamma(b)^2}{3\Gamma(b)^2+4^{1-b}\sqrt{\pi}\Gamma(2b)t}\right)^2\quad (t\geq 0).
  \end{align}
  Consequently, for the arcsine coalescent, that is the beta coalescent with parameters $a=b=\frac{1}{2},$ we find
  \begin{align*}
    c(t) = \left(\frac{3\sqrt{\pi}}{3\sqrt{\pi}+2t}\right)^2\quad (t\geq 0).
  \end{align*}

(2)~In the limiting case $a\to 0$ we recover for the rescaled block counting process the well-known hydrodynamic limit
\[ c(t)=\frac{2}{2+t}\quad (t\geq 0),\]
of Kingman's coalescent, cf.~\cite[Equation (2.15)]{wattis_introduction_2006}.
\end{remark}

\begin{proof}
The jump chain $(J^n_k)_{k\geq 0}$ of $C_n(t) = n^\alpha N_n(t\tau_n)$ has transition probabilities
\begin{align*}
\mu_n(c, c-n^\alpha(l-1))\coloneqq \Prob{J^n_1=c-n^\alpha(l-1) \middle| J^n_0=c} = \begin{cases}
{n^{-\alpha}c\choose l}\frac{\lambda_{n^{-\alpha}c, l}}{\lambda_{n^{-\alpha}c}} & \text{if }c>n^\alpha, 2\leq l\leq n^{-\alpha}c,\\
1 & \text{if }c=n^\alpha, l=1,\\
0 & \text{otherwise.}
\end{cases}
\end{align*}
Denoting by $\lambda_n(c)$ the total rate of $C_n$ in state $c\in E_n$ for any $f\in C^\infty([0, 1])$ the generator of $C_n$ is given by
\begin{align}\label{eq:beta_block_counting_generator}
  \mathcal{G}_nf(c) &= \lambda_n(c)\int_{E_n} (f(c')-f(c))\mu_n(c, dc')\notag\\
  &= \tau_n\lambda_{n^{-\alpha}c}\sum_{l=2}^{n^{-\alpha}c}\left(f(c-n^\alpha(l-1))-f(c)\right)\binom{n^{-\alpha}c}{l}\frac{\lambda_{n^{-\alpha}c, l}}{\lambda_{n^{-\alpha}c}}\notag\\
  &= \tau_n\sum_{l=2}^{n^{-\alpha}c}\left(-n^\alpha(l-1)f'(c)+R_2(\vartheta_{n,l})\right)\binom{n^{-\alpha}c}{l}\lambda_{n^{-\alpha}c, l},
  \end{align}
where we used Taylor's approximation in the third equality. Taylor's approximation ensures the existence of a value $\vartheta_{n,l}\in (c-(l-1)/n, c)$ such that the remainder term $R_2(\vartheta_{n,l})$ is given, for instance, by its Lagrange form
\begin{align*}
R_2(\vartheta_{n,l}) = \frac{1}{2}\left(n^\alpha(l-1)\right)^2 f''(\vartheta_{n,l}).
\end{align*}
First notice that by Lemma \ref{lem:order_of_gamma1}
\begin{align*}
-n^\alpha\tau_n\sum_{l=2}^{n^{-\alpha}c}{n^{-\alpha}c\choose l}(l-1)\lambda_{n^{-\alpha}c, l}=-n^\alpha\tau_n\gamma^{(1)}_{n^{-\alpha}c}\to -\frac{\Gamma(a+b)}{(1-a)(2-a)\Gamma(b)}c^{2-a}
\end{align*}
as $n\to\infty,$ since $\tau_n$ is chosen to be of order $n^{(1-a)\alpha},$ and $a<1$ by assumption. Since $f$ has derivatives of arbitrarily high order on $[0, 1]$, $f''$ attains its supremum $\inorm{f''}_\infty\coloneqq \sup_{x\in [0, 1]}\abs{f''(x)}$. Consequently, as $n\to\infty$ we obtain
\begin{align*}
\tau_n\sum_{l=2}^{n^{-\alpha}c}{n^{-\alpha}c\choose l}\lambda_{n^{-\alpha}c, l}R_2(x_{n,l}) &\leq n^{2\alpha}\frac{\tau_n}{2}\inorm{f''}_\infty\sum_{l=2}^{n^{-\alpha}c}{n^{-\alpha}c\choose l}\lambda_{n^{-\alpha}c, l}(l-1)^2\\
& = n^{2\alpha}\frac{\tau_n}{2}\inorm{f''}_\infty\gamma^{(2)}_{n^{-\alpha}c}\sim \frac{1}{2}c^2\tau_n\inorm{f''}_\infty\\
&\to \begin{cases}
  0 & \text{if }\beta<0,\\
  \frac{1}{2}\inorm{f''}_\infty & \text{if }\beta=0,\\
  \infty & \text{if }\beta>0,
\end{cases}
\end{align*}
where we used Lemma \ref{lem:order_of_gamma2} in the third step. For the term on the left hand side to vanish we need $\beta=(1-a)\alpha<0,$ which explains the restriction $\alpha<0$ for $a<1$. For $a>1$ the only non-trivial solution of $m(t)$ is obtained for $\beta=0$. However, using $\gamma_n^{(3)}\sim\frac{a}{a+b+1}n^3,$ Lemma~\ref{lem:order_of_gamma3}, one can see that in this case the remainder term of third order in Taylor's approximation has Lagrange form $R_3(\vartheta_{n, l})=\frac{1}{2}(n^\alpha(l-1))^3 f'''(\theta_{n, l})$ for some $\theta_{n,l}\in (c-(l-1)/n, c)$,  and one can show that
\begin{align*}
  \tau_n\sum_{l=2}^{n^{-\alpha}c}\binom{n^{-\alpha}c}{l}\lambda_{n^{-\alpha}c, l}R_3(x_{n,l})
\end{align*}
xxx NEED BOUNDED BELOW is bounded above by a term asymptotically equivalent to
  \begin{align*}
    \begin{cases}
      0 & \text{if }\beta<0,\\
      1 & \text{if }\beta=0,\\
      n^\beta & \text{if }\beta>0,
    \end{cases}
  \end{align*}
as $n\to\infty.$ That is, for $a>1$ and for $a<1, \beta=0$ we do not obtain a diffusion limit as $n\to\infty.$

This shows the convergence
\begin{align}\label{eq:convergence_of_generators}
\lim_{n\to\infty}\sup_{c\in E_n}\abs{\gen{G}_nf(c)-\gen{G}f(c)}\to 0,
\end{align}
where the operator $\gen{G}$ is defined by \begin{align}\label{eq:generator_number_of_blocks}
  \mathcal{G}f(c) \coloneqq -\frac{\Gamma(a+b)}{(1-a)(2-a)\Gamma(b)}c^{2-a}f'(c).
\end{align}
Since $[0, 1]\ni c\mapsto -c^{2-a}\Gamma(a+b)/(1-a)(2-a)\Gamma(b)$ is Lipschitz continuous, Theorem 2.1 in Chapter 8 of \cite{EthierKurtz1986} yields that the set $C^\infty([0, 1])$ is a core for $\mathcal G,$ and the closure of $\{(f, \mathcal G f)\colon f\in C^\infty([0, 1])\}$ is single-valued and generates a Feller semigroup $\{T(t)\}$ on $C([0, 1])$. By Theorem 2.7 in Chapter 4 of \cite{EthierKurtz1986} there exists a process $c$ corresponding to $\{T(t)\}$.

To prove that $C_n$ converges in $D_{[0, 1]}([0, \infty))$ in the Skorokhod topology to $c$ as $n\to\infty,$ it suffices by Corollary 8.7 of Chapter 4 to show that \eqref{eq:convergence_of_generators} holds for all $f$ in a core for the generator $\mathcal G,$ which we have just done.

\end{proof}
Instead of the restriction $\Pi_n$ of the beta coalescent $\Pi,$ we now rescale the latter process, namely for each $n\in\mathbb N$ let $C^\star_n=\{C^\star_n(t), t\geq 0\}$ be defined by $C^\star_n(t)\coloneqq n^{\alpha}\#\Pi(t\tau_{n}).$ In particular, notice that the initial state of this process is $C^\star_n(0)=\infty,$ irrespective of $\alpha$ and $\beta.$

\begin{theorem}\label{thm:coalescent_limit}
Let $(\tau_{n})$ be of order $n^\beta$ with $\beta={(1-a)\alpha}$. Then as $n\to\infty$ we have for $a<1$ convergence
\begin{align}
\{C^\star_n(t), t\geq 0\}\to \{c^\star(t), t\geq 0\}
\end{align}
in $D_{\mathbb R}([0, \infty))$ in the Skorokhod topology and the deterministic limit is given by
\begin{align}
 c^\star(t) \coloneqq \left(\frac{\Gamma(a+b)}{(2-a)\Gamma(b)}t\right)^{\frac{1}{a-1}}.
\end{align}
\end{theorem}
Before we turn to the proof of Theorem~\ref{thm:coalescent_limit} notice that its statement suggests the following scaling invariance of the limit $c^\star,$ which is easiliy verified. For any real numbers $\alpha, m$ and $\beta=(1-a)\alpha$ we have
\begin{align}
  m^\alpha c^\star(tm^\beta)=c^\star(t).
\end{align}
\begin{proof}
For the most part the calculations are identical to the ones in the proof of Theorem~\ref{thm:scalar_limit_generator}. Notice that the process $C^\star_n$ has state space $E_n^\star\coloneqq n^\alpha\mathbb N\cup\{\infty\}$ and initial state $C_n(0)=\infty$. Because of the consistency of the $\Lambda$ $n$-coalescents, i.e. $\Pi_n$ is equal in distribution to the restriction of $\Pi$ to $[n]$, the generators of $C_n$ and $C^\star_n$ are of precisely the same form, except that the generator of $C_n$ operates on functions $f$ mapping $n^\alpha[n]$ to $\mathbb R,$ whereas $C^\star_n$ operates on functions $f$ mapping $n^\alpha\mathbb N\cup\{\infty\}$ to $\mathbb R$. For this reason the generator calculations for $C^\star_n$ are identical to the ones given in the proof of Theorem~\ref{thm:scalar_limit_generator}. In particular, the limit  $c^\star(t)$ of $C^\star_n(t)$ as $n\to\infty$ satisfies the ordinary differential equation~\eqref{eq:ode} with boundary condition $c^\star(0)=\infty.$ However, we already solved this ODE in equation~\eqref{eq:ode_solution}.
\end{proof}

\begin{remark}
(1)~Applying Legendre's duplication formula as in the previous remark, we find for $a=\frac{1}{2}$
  \begin{align*}
    c^\star(t) = \frac{9\Gamma(b)^4}{16^{1-b}\pi\Gamma(2b)^2}\frac{1}{t^2}\quad (t\geq 0),
  \end{align*}
  which boils down to
  \begin{align*}
    c^\star(t) = \frac{9}{4}\frac{\pi}{t^2}\quad (t\geq 0)
  \end{align*}
  for the arcsine coalescent.
  
(2)~In the limiting case $a\to 0$ we obtain
\[ c^\star(t)=\frac{2}{t}\quad (t\geq 0),\]
which agrees with the result for Kingman's coalescent.
\end{remark}

\section{Block size spectrum}
For $d\in\mathbb{N}$ let the rescaled block size spectrum $(C_{n,i})_{i=1}^{d+1}\coloneqq (C{n,i}(t), t\geq 0)_{i=1}^{d+1}$ be defined by
\begin{align}
C_{n,i}(t) \coloneqq n^{-1}\type{i}\Pi_n(t\tau_n), i\in [d],\qquad c_{n,d+1}(t)\coloneqq n^{-1}\sum_{i=d+1}^n \type{i}\Pi_n(t\tau_n).
\end{align}
For $l\in\mathbb{N}_0^{d+1}$ with $\abs{l}>1$ we say that an $l$-merger occurs in $\Pi_n$ if among the merging blocks there are $l_1$ singletons, $l_2$ blocks of size $2$, ..., $l_d$ blocks of size $d$ and $l_{d+1}$ blocks of size at least $d+1$.
The process $(C_{n,i})_{i=1}^{d+1}$ has state space $E_n^d\coloneqq n^{-1}\{0, \ldots, n\}^{d+1}\setminus\{0\},$ initial state $(1, 0, \ldots, 0),$ absorbing state $(0, 0, \ldots, n^{-1})$ and evolves according to the following dynamics:
\begin{align}
 \text{a transition }\qquad c\mapsto c-\frac{l-{\rm e}_{\inorm l\wedge (d+1)}}{n}\qquad\text{ occurs at rate }\quad\lambda_{n\abs c, \abs l}\prod_{i=1}^{d+1}\binom{nc_i}{l_i},
\end{align}
if $c\in E_n^d$ and $l_i\leq c_i$ for all $i\in [d+1],$ where $\inorm{l}\coloneqq \sum_{i=1}^{d+1}il_i$ and ${\rm e}_i=(\delta_{ij})_{j=1}^{d+1}$ denotes the $i$th unit vector in $\mathbb R^{d+1}$.

Let $\partial_i=\frac{\partial}{\partial x_i}$ denote the $i$th partial derivative.
\begin{proposition}\label{prop:limit_generator}
Fix $d\in\mathbb{N}.$ For a sequence $(\tau_n)$ of order $n^{a-1}$ and $a<1$ we have convergence
\begin{align*}
(C_{n,1}(t), \ldots, C_{n,d+1}(t))\to (c_1(t), \ldots, c_{d+1}(t)),
\end{align*}
in $D_{[0, 1]^{d+1}}([0, \infty))$ in the Skorokhod topology, where the latter process is deterministic with initial state $(c_1(0), \ldots, c_{d+1}(0)) = (1, 0, \ldots, 0)$ and generator
\begin{align}
	\gen{G}f(c) & \coloneqq\frac{\Gamma(a+b)}{\Gamma(b)} \sum_{i=1}^{d} \left(-\frac{c_i\abs{c}^{1-a}}{1-a}+ \sum_{m=2}^i\risfac{a}{m-2}\abs{c}^{2-a-m}\sum_{\substack{l\in\mathbb{N}_0^{d}\\ \abs l=m, \inorm l=i}}    \prod_{k=1}^{d}\frac{c_k^{l_k}}{l_k!}\right)\partial_i f(c)\\\notag
& +\frac{\Gamma(a+b)}{\Gamma(b)}\left(-\frac{\abs{c}^{2-a}}{2-a}+\sum_{r=1}^{d+1}\sum_{m=2}^r\risfac{a}{m-2}\sum_{\substack{l\in\mathbb{N}_0^{d+1}\\ \abs l=m, \inorm l=r}}\prod_{k=1}^{d+1}\frac{c_k^{l_k}}{l_k!}\right)\partial_{d+1}f(c).
\end{align}
\end{proposition}
\begin{proof}
  For a function $f:\mathbb{R}^{d+1}\to\mathbb{R}$ and a vector $\kappa\in\mathbb{N}_0^{d+1}$ let $D^\kappa \coloneqq \partial_1^{\kappa_1}\cdots \partial_{d+1}^{\kappa_{d+1}}$ and $f^{(\kappa)}(x) \coloneqq \partial_1^{\kappa_1}\cdots \partial_{d+1}^{\kappa_{d+1}}f(x)$. Moreover, let $\kappa!\coloneqq \prod_{i=1}^{d+1} \kappa_i!$ and for any vector $x\in\mathbb{R}^{d+1}$ let $x^\kappa\coloneqq \prod_{i=1}^{d+1} x_i^{\kappa_i}$. Letting $\lambda_n(c)$ denote the total rate of $(c_{n,i})_{i=1}^n$ in state $c\in E_n^d$. Using a Taylor expansion we obtain for the generator $\gen{G}_n$ of $(c_{n,i})_{i=1}^{d+1}$
\begin{align}\label{eq:generator_n}
  \gen{G}_{n}f(c) &\coloneqq \lambda_n(c)\int(f(c')-f(c))\mu(c, dc')\notag\\
  &= \tau_n\lambda_{n\abs{c}}\sum_{\substack{l\in\mathbb{N}_0^{d+1}\\ \abs{l}>1, l\leq nc}}\left(f(c-(l-{\rm e}_{\inorm{l}\wedge (d+1)})/n)-f(c)\right)\frac{\lambda_{n\abs{c}, \abs{l}}}{\lambda_{n\abs{c}}}\prod_{i=1}^{d+1}{nc_i\choose l_i}\notag\\
  &= \tau_n \sum_{l\in\mathbb{N}_0^{d+1}, \abs l>1} \bigg( -\sum_{\kappa\in\mathbb{N}_0^{d+1}, \abs\kappa=1}\frac{(l-{\rm e}_{\inorm l\wedge (d+1)})^\kappa}{n} D^\kappa f(c)\\
  &\qquad+\sum_{\kappa\in\mathbb{N}_0^{d+1}, \abs\kappa=2} \frac{(l-{\rm e}_{\inorm l\wedge (d+1)})^\kappa}{n^2\kappa!} D^\kappa f\left(c-\vartheta_{n,l}\frac{l-{\rm e}_{\inorm l\wedge (d+1)}}{n}\right)\bigg)\lambda_{n\abs c, \abs l}\prod_{i=1}^{d+1}{nc_i\choose l_i}\notag,
\end{align}
for any $f\in C^\infty([0, 1]^{d+1}),$ $c\in E_n^d$ and for some $\vartheta_{n,l}\in [0, 1]^{d+1}.$
Let $\supnorm{f}\coloneqq \sup_{x\in [0, 1]^{d+1}} \abs{f(x)}$.

{\em Part 1.} Let us first consider summands corresponding to $\abs\kappa=2$ by focusing on
\begin{align}
  T^{(2)}(n) &\coloneqq \frac{\tau_n}{n^2} \sum_{\substack{l\in\mathbb{N}_0^{d+1}\\ \abs l>1}} \sum_{\substack{\kappa\in\mathbb{N}_0^{d+1}\\ \abs\kappa=2}} \frac{(l-{\rm e}_{\inorm l\wedge (d+1)})^\kappa}{\kappa!} D^\kappa f\left(c-\vartheta_{n,l}\frac{l-{\rm e}_{\inorm l\wedge (d+1)}}{n}\right)\\
  &\qquad\qquad\qquad\qquad\times\lambda_{n\abs c, \abs l}\prod_{i=1}^{d+1}{nc_i\choose l_i}\notag,
\end{align}
Evidently, in this case there exist (possibly equal) $i, j\in [d+1]$ with $\kappa={\rm e}_i+{\rm e}_j$. Notice that
\begin{align*}
(l-{\rm e}_{\inorm l\wedge (d+1)})^\kappa
&\leq l^\kappa \leq\begin{cases}
l_i^2 &\text{if }\kappa=2{\rm e}_i\text{ for some } i\in [d+1],\\
l_il_j&\text{if }\kappa={\rm e}_i+{\rm e}_j\text{ for some }i, j\in [d+1], i\neq j.
\end{cases}
\end{align*}
Hence, for fixed $i\in [d+1]$ and $\kappa=2{\rm e}_i$
\begin{align}
T^2(n) &=\tau_n \sum_{\substack{l\in\mathbb{N}_0^{d+1}\\ \abs l>1}}\frac{1}{2n^2}(l-{\rm e}_{\inorm l\wedge (d+1)})^{2{\rm e}_i} D^{2{\rm e}_i} f\left(c-\vartheta_{n,l}\frac{l-{\rm e}_{\inorm l\wedge (d+1)}}{n}\right)\lambda_{n\abs c, \abs l}\prod_{k=1}^{d+1}{nc_k\choose l_k}\notag\\
&\leq \supnorm{f^{(2{\rm e}_i)}}\frac{\tau_n}{n^2}\sum_{\substack{l\in\mathbb{N}_0^{d+1}\\\abs l>1}} l_i^2 \lambda_{n\abs c, \abs l}\prod_{k=1}^{d+1}{nc_k\choose l_k}\notag\\
&= \supnorm{f^{(2{\rm e}_i)}}\frac{\tau_n}{\risfac{(a+b)}{n\abs c-2}n^2}\sum_{\substack{l\in\mathbb{N}_0^{d+1}\\\abs l>1}} l_i^2 \risfac{a}{\abs l-2}\risfac{b}{n\abs c-\abs l}\prod_{k=1}^{d+1}{nc_k\choose l_k}\\
&= \frac{\supnorm{f^{(2{\rm e}_i)}}}{(1-a)(2-a)}\frac{\tau_n}{\risfac{(a+b)}{n\abs c-2}n^2}\sum_{\substack{l\in\mathbb{N}_0^{d+1}\\\abs l>1}} l_i^2 \risfac{(a-2)}{\abs l}\risfac{b}{n\abs c-\abs l}\prod_{k=1}^{d+1}{nc_k\choose l_k}.\notag
\end{align}
Writing $l_i^2=\fallfac{l_i}{2}+l_i,$ we find
\begin{align*}
&\frac{\tau_n}{\risfac{(a+b)}{n\abs c-2}n^2}\sum_{\substack{l\in\mathbb{N}_0^{d+1}\\\abs l>1}} \fallfac{l_i}{2} \risfac{(a-2)}{\abs l}\risfac{b}{n\abs c-\abs l}\prod_{k=1}^{d+1}{nc_k\choose l_k}\\
&=\frac{\tau_n}{\risfac{(a+b)}{n\abs c-2}n^2}(1-a)(2-a)nc_i(nc_i-1)\risfac{(a+b)}{n\abs c-2}\\
&\sim (a-2)(a-1)c_i^2n^{a-1}\to 0,
\end{align*}
as $n\to\infty,$ and
\begin{align}\label{eq:first_summand_asymptotics}
&\frac{\tau_n}{\risfac{(a+b)}{n\abs c-2}n^2}\sum_{\substack{l\in\mathbb{N}_0^{d+1}\\\abs l>1}} l_i\risfac{(a-2)}{\abs l}\risfac{b}{n\abs c-\abs l}\prod_{k=1}^{d+1}{nc_k\choose l_k}\notag\\
&=(a-2)\frac{\tau_n}{\risfac{(a+b)}{n\abs c-2}n^2}\left(nc_i\risfac{(a+b-1)}{n\abs c-1}-\risfac{b}{n\abs c-1}\right)\\
&\sim (a-2)\Gamma(a+b)\left(\frac{c_in^{a-2}}{\Gamma(a+b-1)}-\frac{\abs{c}^{1-a}n^{-2}}{\Gamma(b)}\right)\to 0,\notag
\end{align}
as $n\to\infty,$ where we applied Lemma \ref{lem:super-lemma} (with $k=2{\rm e}_i$ in the first case and $k={\rm e}_i$ in the second) and Lemma \ref{lem:risfac_ratio_asymptotics}.

For fixed $i,j\in [d+1]$ such that $i\neq j$ and $\kappa={\rm e}_i+{\rm e}_j$ we obtain
\begin{align*}
&\frac{\tau_n}{n^2}\sum_{l\in\mathbb{N}_0^{d+1}, \abs l>1}l_il_jD^{{\rm e}_i+{\rm e}_j}f\left(c-\vartheta_{n,l}\frac{l-{\rm e}_{\inorm l\wedge (d+1)}}{n}\right)\lambda_{n\abs c, \abs l}\prod_{k=1}^{d+1}{nc_k\choose l_k}\\
&\leq \frac{\supnorm{f^{({\rm e}_i+{\rm e}_j)}}}{(1-a)(2-a)}\frac{\tau_n}{n^2\risfac{(a+b)}{n\abs c-2}}\sum_{\substack{l\in\mathbb{N}_0^{d+1}\\ \abs l>1}}l_il_j\risfac{(a-2)}{\abs l}\risfac{b}{n\abs c-\abs l}\prod_{k=1}^{d+1}{nc_k\choose l_k}\\
&= \frac{\supnorm{f^{({\rm e}_i+{\rm e}_j)}}}{(1-a)(2-a)}\frac{\tau_n}{n^2\risfac{(a+b)}{n\abs c-2}}(a-2)(a-1)c_ic_jn^2\risfac{(a+b)}{n\abs c-2}\\
&=\supnorm{f^{({\rm e}_i+{\rm e}_j)}}c_ic_jn^{a-1}\to 0,
\end{align*}
as $n\to\infty,$ where we applied Lemmata \ref{lem:risfac_ratio_asymptotics} and \ref{lem:super-lemma} (with $k={\rm e}_i+{\rm e}_j$). Summarizing, we showed that $T^2(n)$ vanishes as $n\to\infty$.

{\em Part 2.} We now focus on $\abs\kappa=1,$ i.e.~we consider
\begin{align*}
  T^{(1)}(n) &\coloneqq -\frac{\tau_n}{n} \sum_{\substack{l\in\mathbb{N}_0^{d+1}\\ \abs l>1}} \sum_{\substack{\kappa\in\mathbb{N}_0^{d+1}\\ \abs\kappa=1}}(l-{\rm e}_{\inorm l\wedge (d+1)})^\kappa D^\kappa f(c)\lambda_{n\abs c, \abs l}\prod_{k=1}^{d+1}{nc_k\choose l_k}\\
  &= -\frac{\tau_n}{n\risfac{(a+b)}{n\abs c-2}}\sum_{i=1}^{d+1}\partial_i f(c)\sum_{\substack{l\in\mathbb{N}_0^{d+1}\\ \abs l>1}} (l_i-1_{\{i=\inorm l\wedge (d+1)\}})\risfac{a}{\abs l-2}\risfac{b}{n\abs c-\abs l}\prod_{k=1}^{d+1}{nc_k\choose l_k}.
\end{align*}

Now consider in $T^{(1)}(n)$ the summands corresponding to a fixed $i\in [d+1].$ We partition these summands and analyse their asymptotics seperately as follows. Firstly, by Lemma \ref{lem:super-lemma} (with $k={\rm e}_i$) we have that 
\begin{align}
  O_{i}(n) & \coloneqq\frac{\tau_n}{n\risfac{(a+b)}{n\abs c-2}}\sum_{\substack{l\in\mathbb{N}_0^{d+1}\\ \abs l>1}} l_i\risfac{a}{\abs l-2}\risfac{b}{n\abs c-\abs l}\prod_{k=1}^{d+1}{nc_k\choose l_k}\notag\\
&=\frac{1}{(a-1)(a-2)}\frac{\tau_n}{n\risfac{(a+b)}{n\abs c-2}}\sum_{\substack{l\in\mathbb{N}_0^{d+1}\\ \abs l>1}} l_i\risfac{(a-2)}{\abs l}\risfac{b}{n\abs c-\abs l}\prod_{k=1}^{d+1}{nc_k\choose l_k}\notag\\
&=\frac{1}{(a-1)(a-2)}\frac{\tau_n}{n\risfac{(a+b)}{n\abs c-2}}\left ((a-2)nc_i\risfac{(a+b-1)}{n\abs c-1}-(a-2)\risfac{b}{n\abs c-1}nc_i\right)\\
&= \frac{c_i}{a-1}\frac{\tau_n}{\risfac{(a+b)}{n\abs c-2}}\left (\risfac{(a+b-1)}{n\abs c-1}-\risfac{b}{n\abs c-1}\right)\notag\\
&\sim \frac{\Gamma(a+b)}{a-1}c_in^{a-1}\left(\frac{1}{\Gamma(a+b-1)}-\frac{(\abs cn)^{1-a}}{\Gamma(b)}\right)\sim -\frac{\Gamma(a+b)}{(a-1)\Gamma(b)}c_i{\abs c}^{1-a}\notag,
\end{align}
as $n\to\infty.$ Secondly, for any fixed $i\in [d+1]$ the summand corresponding to the indicator $1_{\{i=\inorm l \wedge (d+1)\}}$ has asymptotic behaviour
\begin{align}
  I_{i}(n) &\coloneqq -\frac{\tau_n}{n\risfac{(a+b)}{n\abs c-2}} \sum_{\substack{l\in\mathbb{N}_0^{d+1}\\ \abs l>1}}1_{\{i=\inorm l\wedge (d+1)\}}\risfac{a}{\abs l-2}\risfac{b}{n\abs c-\abs l}\prod_{k=1}^{d+1}{nc_k\choose l_k}\notag\\
  &= -\frac{\tau_n}{n\risfac{(a+b)}{n\abs c-2}}\sum_{\substack{l\in\mathbb{N}_0^{d+1}\\ \abs l>1, \inorm l=i}}\risfac{a}{\abs l-2}\risfac{b}{n\abs c-\abs l}\prod_{k=1}^{d+1}{nc_k\choose l_k}\notag\\
&= -\frac{\tau_n}{n\risfac{(a+b)}{n\abs c-2}}\sum_{m=2}^{n\abs c\wedge i}\risfac{a}{m-2}\risfac{b}{n\abs c-m}\sum_{\substack{l\in\mathbb{N}_0^{d+1}\\ \abs l=m, \inorm l=i}}\prod_{k=1}^{d+1}{nc_k\choose l_k},\notag\\
&\sim -\frac{\Gamma(a+b)}{\Gamma(b)}n^{a-2}\sum_{m=2}^i\risfac{a}{m-2}\abs{c}^{2-a-m}n^{-a-(m-2)}n^m\sum_{\substack{l\in\mathbb{N}_0^{d+1}\\ \abs l=m, \inorm l=i}}\prod_{k=1}^{d+1}\frac{c_k^{l_k}}{l_k!}\notag\\
&= -\frac{\Gamma(a+b)}{\Gamma(b)}\sum_{m=2}^i\risfac{a}{m-2}\abs{c}^{2-a-m}\sum_{\substack{l\in\mathbb{N}_0^{d+1}\\ \abs l=m, \inorm l=i}}\prod_{k=1}^{d+1}\frac{c_k^{l_k}}{l_k!},\notag
\end{align}
as $n\to\infty.$

However, $I_{d+1}(n)$ must be treated as a special case. Using the set equality
\begin{align*}
\{l\in\mathbb{N}_0^{d+1}\colon \abs l>1, \inorm l\geq d+1\} &= \mathbb{N}_d^{d+1}\setminus \\
&\left(\{l\in\mathbb{N}_0^{d+1}\colon \abs l\geq 2, \inorm l\in [d]\}\cup\{l\in\mathbb{N}_0^{d+1}\colon 0\leq \abs l\leq 1\}\right),
\end{align*}
it follows that
\begin{align*}
I_{d+1}(n)&\coloneqq \frac{\tau_n}{(1-a)(2-a)n\risfac{(a+b)}{n\abs c-2}}\Bigg( \sum_{l\in\mathbb{N}_0^{d+1}}\risfac{(a-2)}{\abs l}\risfac{b}{n\abs c-l}\prod_{k=1}^{d+1}{nc_k\choose l_k}\\
&-\sum_{r=1}^d\sum_{m=2}^r\risfac{(a-2)}{m}\risfac{b}{n\abs c-m}\sum_{\substack{l\in\mathbb{N}_0^{d+1}\\ \abs l=m, \inorm l=r}}\prod_{k=1}^{d+1}{nc_k\choose l_k}\\
&- \sum_{\substack{l\in\mathbb{N}_0^{d+1}\\ \abs l=1}}\risfac{(a-2)}{l}\risfac{b}{n\abs c-l}\prod_{k=1}^{d+1}{nc_k\choose l_k}-\risfac{b}{n\abs c}\Bigg).
\end{align*}
We now study each of the summands in $I_{d+1}(n).$ Using Lemma \ref{lem:super-lemma} (with $k=0$) we find that the first summand
\begin{align*}
&\frac{\tau_n}{(1-a)(2-a)n\risfac{(a+b)}{n\abs c-2}}\sum_{l\in\mathbb{N}_0^{d+1}}\risfac{(a-2)}{\abs l}\risfac{b}{n\abs c-l}\prod_{k=1}^{d+1}{nc_k\choose l_k}\\
&= \frac{\tau_n}{(1-a)(2-a)n\risfac{(a+b)}{n\abs c-2}}\risfac{(a+b-2)}{n\abs c}\\
&\sim \frac{\Gamma(a+b)}{(1-a)(2-a)\Gamma(a+b-2)}n^{a-2}
\end{align*}
vanishes as $n\to\infty.$ For the second summand, we find
\begin{align*}
&-\frac{\tau_n}{(1-a)(2-a)n\risfac{(a+b)}{n\abs c-2}}\sum_{r=1}^d\sum_{m=2}^r\risfac{(a-2)}{m}\risfac{b}{n\abs c-m}\sum_{\substack{l\in\mathbb{N}_0^{d+1}\\ \abs l=m, \inorm l=r}}\prod_{k=1}^{d+1}{nc_k\choose l_k}\\
&\sim -\frac{\Gamma(a+b)}{\Gamma(b)}\sum_{r=1}^d\sum_{m=2}^r \risfac{a}{m-2}\sum_{\substack{l\in\mathbb{N}_0^{d+1}\\ \abs l=m, \inorm l=r}}\prod_{k=1}^{d+1}\frac{c_k^{l_k}}{l_k!},
\end{align*}
as $n\to\infty.$

Using
\begin{align*}
\sum_{\substack{l\in\mathbb{N}_0^{d+1}\\ \abs l=1}} \risfac{(a-2)}{\abs l}\risfac{b}{n\abs c-l}\prod_{k=1}^{d+1}{nc_k\choose l_k} = \sum_{m=1}^{d+1}(a-2)\risfac{b}{n\abs c-l}nc_m=(a-2)\risfac{b}{n\abs c-l}n\abs c,
\end{align*}
as $n\to\infty$ we find for the final summand
\begin{align*}
-\frac{\tau_n}{(1-a)(2-a)n\risfac{(a+b)}{n\abs c-2}}\sum_{\substack{l\in\mathbb{N}_0^{d+1}\\ \abs l=1}}\risfac{(a-2)}{l}\risfac{b}{n\abs c-l}\prod_{k=1}^{d+1}{nc_k\choose l_k}-\risfac{b}{n\abs c}\sim \frac{\Gamma(a+b)}{(2-a)\Gamma(b)}\abs{c}^{2-a}
\end{align*}
as $n\to\infty.$

Since, by definition, $T^{(1)}(n)=-\sum_{i=1}^{d}(O_i(n)+I_i(n))\partial_i f(c)$, we obtain the convergence
\begin{align}
\lim_{n\to\infty}\sup_{c\in E^d_n}\abs{\mathcal G_nf(c)-\mathcal Gf(c)}=0,
\end{align}
for all $f\in C_c^\star([0, 1]^{d+1}).$

Since $\gen G$ operates on real functions defined on the bounded domain $E^d\coloneqq [0, 1]^{d+1}$ whose boundary is not smooth, a direct analysis of the corresponding semigroup, respectively process, as done in the one-dimensional case in the proof of Theorem \ref{thm:scalar_limit_generator}, is nontrivial, cf.~\cite{Ulmet1992}. Instead, we proceed via the theory of martingale problems. We say that a process is a martingale, if it is a martingale with respect to its natural filtration.

Since $\gen G_n$ is the generator of a Markov jump process,
\begin{align*}
M_n(t)\coloneqq f((c_{n,i}(t))_{i=1}^{d+1}) - \int_0^t \gen G_n f((c_{n,i}(s))_{i=1}^{d+1})ds
\end{align*}
is a martingale for each $f\in B(E^d)$ with compact support. 
Hence, if some subsequence of $\{(c_i^n)_{i=1}^{d+1}, m\geq 2\}$ converges in distribution to $(c_i)_{i=1}^{d+1},$ then for each $f\in C_c^2(E^d)$
\begin{align*}
  f((c_i(t))_{i=1}^{d+1}) - \int_0^t \gen G_n f((c_i(s))_{i=1}^{d+1})ds
\end{align*}
is a martingale by the continuous mapping theorem (cf.~Corollary 1.9 of Chapter 3 in \cite{EthierKurtz1986}) and Problem 7 of Chapter 7 in \cite{EthierKurtz1986}, since $(c_i^n)_{i=1}^{d+1}$ is bounded by $1$ and $M_n(t)$ is uniformly integrable, and so $(c_i)_{i=1}^{d+1}$ is a solution of the martingale problem for $\{(f, \mathcal Gf)\colon f\in C_c^2(E)\}$. Once we show that the function $b=(b_i)_{i=1}^{d+1}$ from $ [0, \infty)\times\mathbb R^{d+1}$ to $\mathbb R^{d+1},$ defined for $t\geq 0,$ $c\in\mathbb{R}^{d+1}$ by
\begin{align*}
b_i(t, c)\coloneqq b_i(c)\coloneqq \frac{\Gamma(a+b)}{\Gamma(b)}\begin{cases*}
-\frac{c_i\abs c}{1-a}+\sum_{m=2}^i \risfac{a}{m-2}\abs c^{2-a-m}\sum_{\substack{l\in\mathbb N^d\\ \abs l=m, \inorm l=i}} \prod_{k=1}^d \frac{c_k^{l_k}}{l_k!} & if $c\in [0,1]^{d+1}, i\in [d],$\\
-\frac{\abs{c}^{2-a}}{2-a}+\sum_{r=1}^{d+1}\sum_{m=2}^r\risfac{a}{m-2}\sum_{\substack{l\in\mathbb{N}_0^{d+1}\\ \abs l=m, \inorm l=r}}\prod_{k=1}^{d+1}\frac{c_k^{l_k}}{l_k!} & if $c\in [0,1]^{d+1}, i=d+1,$\\
0 & otherwise
\end{cases*}
\end{align*}
satisfies the conditions of Theorem 3.10 of Chapter 5 in \cite{EthierKurtz1986}, then Theorem 2.6 of Chapter 8 implies that the martingale problem for $\{(f, \gen G f)\colon f\in C_c^2(E) \}$ is well-posed. This is indeed the case, since
\begin{align*}
cb(c) &= \sum_{i=1}^{d+1} c_ib_i(c),
\end{align*}
and moreover, using $c_i\leq \abs c\leq 1$
\begin{align*}
\sum_{i=1}^d c_ib_i(c)\leq K\sum_{i=1}^d c_i \abs c^{2-a}K_{d,i}\leq \tilde K_d\abs c^{3-a}\leq \hat K_d\abs c^2
\end{align*}
and
\begin{align*}
c_{d+1}b_{d+1}(c)\leq \abs c \sum_{r=1}^{d+1}\sum_{m=2}^r \abs c^m K_{d,r,m}\leq K_d\abs c^3\leq K_d \abs c^2,
\end{align*}
where the $K, $ $K_{d, i},$ $K_d,$ $\tilde K_d,$ $\hat K_d,$ $K_{d,r,m}$ denote suitable constants. It is now straightforward to verify the conditions of Corollary 8.16 of Chapter 4 in \cite{EthierKurtz1986} which implies the convergence in the statement.
\end{proof}

Proposition \ref{prop:limit_generator} implies that for each $i\in\mathbb{N}$ $c_i(t)$ is a solution of the ODE
\begin{align}\label{eq:bss_ode}
  c_i'(s) - \frac{\Gamma(a+b)}{(a-1)\Gamma(b)}c_i(s)c(s)^{1-a} &= \frac{\Gamma(a+b)}{\Gamma(b)} \sum_{m=2}^i \risfac{a}{m-2}c(s)^{2-a-m}\sum_{\substack{l\in\mathbb{N}_0^d\\ \abs l=m, \inorm l =i}} \prod_{k=1}^d \frac{c_k^{l_k}(s)}{l_k!}, 
\end{align}
or
\begin{align}
c_i'(s) &= \frac{\Gamma(a+b)}{(a-1)(a-2)\Gamma(b)}\sum_{l\in\mathbb{N}_0^i, \inorm l=i} \risfac{(a-2)}{\abs l}c(s)^{2-a-\abs l} \prod_{k=1}^i \frac{c_k(s)^{l_k}}{l_k!}\\\notag
&= \frac{G}{i!}\sum_{m=1}^i \risfac{(a-2)}{m}c(t)^{2-a-m}\sum_{\substack{l\in\mathbb{N}_0^d\\ \abs l=m, \inorm l=i}} \frac{i!}{\prod_{k=1}^i l_k!(k!)^{l_k}}\prod_{k=1}^i (k!c_k(t))^{l_k}\\\notag
&= \frac{G}{i!}\sum_{m=1}^i \risfac{(a-2)}{m}c(t)^{2-a-m} B_{i,m}(w_\bullet)\\\notag
&= \frac{G}{i!}c(t)^{2-a}B_i(v_\bullet, w_\bullet),
\end{align}
where
\begin{align}
G\coloneqq \Gamma(a+b)/(1-a)(2-a)\Gamma(b)
\end{align}
and $v_\bullet=(v_k),$ $w_\bullet=(w_k)$ are sequences defined by $v_k\coloneqq \risfac{(a-2)}{k}c(t)^{-k}$ and $w_k\coloneqq k!c_k(t)$. Consider now the generating function
\begin{align}\label{eq:generating_function}
\mathscr{G}(t, x)\coloneqq \sum_{i\geq 1} c_i(t)x^i\quad (x\in [-1, 1], t\geq 0).
\end{align}
We write $\partial_t$ for the partial derivative $\frac{\partial}{\partial t}$ with respect to $t$.
\begin{lemma}
The generating function $\mathscr{G}$ solves the partial differential equation
\begin{align}\label{eq:pde}
\partial_t \mathscr{G}(t, x) &= -\frac{\Gamma(a+b)}{(1-a)(2-a)\Gamma(b)} ((c(t)-\mathscr{G}(t, x))^{2-a}-c(t)^{2-a})\quad (x\in (-1, 1), t\geq 0)
\end{align}
with boundary condition $\mathscr{G}(0, x)=x$ for $x\in [-1, 1]$.
\end{lemma}
\begin{proof}
We use the well known fact, cf.~\cite[Equation (1.11)]{Pitman1999}, that for any two sequences $(v_k), (w_k),$ the exponential generating function of the associated complete Bell polynomials $(B_k(v_\bullet, w_\bullet))$ is given by
\begin{align*}
\sum_{k\geq 1} B_{k}(v_\bullet, w_\bullet)\frac{x^k}{k!} = v(w(x)),
\end{align*}
where these quantities are defined, and $v,$ respectively $w,$ denotes the exponential generating function of $v_\bullet=(v_k),$ respectively $w_\bullet=(w_k),$ i.e.~$v(\theta)\coloneqq \sum_{k\geq 1} v_k\theta^k/k!,$ $w(x)\coloneqq \sum_{k\geq 1}w_k x^k/k!.$
For our particular choice of $(v_k)$ and $(w_k)$, we find
\begin{align*}
v(\theta) &\coloneqq \sum_{j\geq 1} v_j\frac{\theta^j}{j!} = \sum_{j\geq 1}\risfac{(a-2)}{j}\frac{\theta^j}{j!c(t)^j} = (1-\frac{\theta}{c(t)})^{2-a}-1,\\
w(x) &\coloneqq \sum_{k\geq 1} w_k\frac{x^k}{k!}=\sum_{k\geq 1} c_k(t)x^k = \mathscr G(t, x),
\end{align*}
for $\abs{\theta}<c(t),$ $x\in [-1, 1].$ Noticing that $\abs{\mathscr{G}(t, x)}<c(t)$ for $\abs x<1$ we obtain
\begin{align*}
\partial_t \mathscr{G}(t, x) = \sum_{i\geq 1}c'_i(t) x^i &= G c(t)^{2-a}\sum_{i\geq 1}B_i(\risfac{(a-2)}{\bullet}c(t)^{-\bullet}, \bullet!c_{\bullet}(t))\frac{x^i}{i!}\\
&= G c(t)^{2-a}v(\mathscr{G}(t, x))= G ((c(t)-\mathscr{G}(t, x))^{2-a}-c(t)^{2-a}).
\end{align*}
\end{proof}

\begin{theorem}\label{thm:generating_function}
The generating function $\mathscr{G}$ is given by
\begin{align}
\mathscr{G}(t, x) &= c(t)-\left((1-x)^{a-1}+\frac{\Gamma(a+b)}{(2-a)\Gamma(b)}t\right)^{\frac{1}{a-1}}\qquad (x\in (-1, 1), t\geq 0).
\end{align}
\end{theorem}
\begin{proof}
First, for fixed $x\in (0, 1)$ consider the transformation
\begin{align*}
g(t, x)\coloneqq c(t)-\mathscr{G}(t, x)\quad (t\geq 0).
\end{align*}
It is straightforward to verify that $g$ solves the Bernoulli differential equation
\begin{align}\label{eq:pde_for_transform}
\partial_t g(t, x) = -\frac{\Gamma(a+b)}{(1-a)(2-a)\Gamma(b)}g(t)^{2-a},
\end{align}
with boundary condition $g(0, x)=1-x$. Notice the remarkable similarity between this partial differential equation and the ordinary differential equation in \eqref{eq:ode} for the total number of blocks. We interpret this as a form of self-similarity in terms of generating functions. It is straightforward to solve \eqref{eq:pde_for_transform} and obtain
\begin{align}
g(t, x) = \left((1-x)^{a-1}+\frac{\Gamma(a+b)}{(2-a)\Gamma(b)}t\right)^{\frac{1}{a-1}}.
\end{align}
\end{proof}

\begin{corollary}\label{cor:limit_of_spectrum}
For the deterministic limit $\{(c_1(t), \ldots, c_d(t)), t\geq 0\}$ we have
\begin{align}
c_i(t) = \frac{c(t)^{2-a}}{i!}B_i\left(\risfac{\left(\frac{1}{1-a}\right)}{\bullet}(-c(t)^{1-a})^{\bullet-1}, \risfac{(1-a)}{\bullet}\right)\quad (i\in [d], t\geq 0).
\end{align}
\end{corollary}
\begin{proof}
From the definition \eqref{eq:generating_function} of $\mathscr G$ it is clear that we can compute its coefficient $c_i(t)$ for instance by evaluating its $i$th partial derivative with respect to $x$ at $x=0.$ To this end, it will prove useful to write $\mathscr G$ as a composition, namely
\begin{align*}
\mathscr G(t, x) = c(t) - (f\circ g)(x),
\end{align*}
where
\[ f(x)\coloneqq \left(x +\frac{\Gamma(a+b)}{(2-a)\Gamma(b)}t\right)^{\frac{1}{a-1}}\quad\text{ and }\quad g(x)\coloneqq (1-x)^{a-1}.\]
We can now find a formula for the $i$th partial derivative of $\mathcal G$ by an application of Fa\`a di Bruno's formula, cf. \cite{Johnson2002}, which states that
\begin{align}\label{eq:faa_di_bruno}
\frac{d^i}{dx^i}(f\circ g)(x) &= \sum_{\pi\in\pset{[i]}} f^{(\card \pi)}(g(x))\prod_{B\in\pi}g^{(\card B)}(x),
\end{align}
for any two real functions $f, g$ that are at least $i$ times differentiable, where $f^{(j)}$ denotes the $j$th derivative of $f$. In our case, for $j\in\mathbb N$ the $j$th derivatives are
\begin{align}
f^{(j)}(x) &= \fallfac{\left(\frac{1}{a-1}\right)}{j}\left(x+ \frac{\Gamma(a+b)}{(2-a)\Gamma(b)}t\right)^{\frac{1}{a-1}-j},
\end{align}
and
\begin{align}
g^{(j)}(x)=(-1)^j\fallfac{(a-1)}{j}(1-x)^{a-1-j}.
\end{align}
Plugging this into \eqref{eq:faa_di_bruno} yields
\begin{align*}
c_i(t) &= -\frac{1}{i!}\partial_x^i\biggr|_{x=0} \mathscr G(t, x)\\
&= -\frac{1}{i!}\sum_{\pi\in\pset{[i]}}\fallfac{\left(\frac{1}{a-1}\right)}{\card\pi}\left(1+ \frac{\Gamma(a+b)}{(2-a)\Gamma(b)}t\right)^{\frac{1}{a-1}-\card\pi}\prod_{B\in\pi}(-1)^{\card B}\fallfac{(a-1)}{\card B}\\
&= -\frac{c(t)}{i!}\sum_{\pi\in\pset{[i]}}(-1)^{\card\pi}\fallfac{\left(\frac{1}{a-1}\right)}{\card\pi}c(t)^{\card\pi(1-a)}\prod_{B\in\pi}\risfac{(1-a)}{\card B}\\
&= \frac{c(t)^{2-a}}{i!}\sum_{m=1}^i \risfac{\left(\frac{1}{1-a}\right)}{m}(-c(t)^{1-a})^{m-1} B_{i,m}(\risfac{(1-a)}{\bullet})\\
&= \frac{c(t)^{2-a}}{i!}B_i\left(\risfac{\left(\frac{1}{1-a}\right)}{\bullet} (-c(t)^{1-a})^{\bullet-1}, \risfac{(1-a)}{\bullet}\right).
\end{align*}
\end{proof}

\begin{corollary}
In the limiting case $a\to 0$ we find for the limiting frequencies of blocks of size $i,$
\begin{align}
  c_i(t) &= c(t)^{2}(1-c(t))^{i-1}=\left(\frac{2}{2+t}\right)^2\left(\frac{t}{2+t}\right)^{i-1}\quad (t\geq 0).
\end{align}
These are the limiting frequencies in the Kingman coalescent in agreement with Smoluchowski's result, cf.~\cite[Table 2]{Aldous1999}.
\end{corollary}
\begin{proof}
As $a\to 0,$ Corollary~\ref{cor:limit_of_spectrum} yields
\begin{align*}
  c_i(t) &= \frac{c(t)^{2-a}}{i!}B_i(\risfac{(\frac{1}{1-a})}{\bullet}(-c(t)^{1-a})^{\bullet-1}, \risfac{(1-a)}{\bullet})\to\frac{c(t)^2}{i!}B_i(\bullet!(-c(t))^{\bullet-1}, \bullet!).
\end{align*}
Moreover, recall that $B_{i, k}(\bullet!)=\lah{i}{k}\coloneqq\binom{i-1}{k-1}\frac{i!}{k!}$ is the $(i, k)$th (unsigned) Lah number, cf.~\cite[Equation (1.55)]{Pitman2006}, counting the number of partitions into $k$ linearly ordered subsets of a set containing $i$ elements. Thus
\begin{align*}
  B_i(\bullet!(-c(t))^{\bullet-1}, \bullet!) = \sum_{k=1}^i k!(-c(t))^{k-1}B_{i, k}(\bullet!)&=\sum_{k=1}^i i!\binom{i-1}{k-1}(-c(t))^{k-1}\\
  &= i!(1-c(t))^{i-1},
\end{align*}
and the claim follows.
\end{proof}
In complete analogy to our discussion of the block counting process of $\Pi_n$, define the process $(C_{n, 1}^\star(t), \ldots, C_{n, d+1}^\star(t), t\geq 0)$ via $C_{n, i}^\star(t)\coloneqq n^{-1}\type{i}\Pi(t\tau_{n})$ for $i\in [d]$ and $C_{n, d+1}^\star(t)=n^{-1}\sum_{i\geq d+1}\type{i}\Pi(t\tau_{n}).$ In particular, the process $(C_{n, 1}^\star, \ldots, C_{n, d+1}^\star)$ has initial state $(\infty, 0, \ldots, 0)$.
\begin{theorem}\label{thm:limit_generator_infty}
Fix $d\in\mathbb{N}.$ For any sequence $(\tau_n)$ such that $\tau_n\sim n^{a-1}$ as $n\to\infty$ and $a<1$ we have convergence
\begin{align*}
(C_{n,1}^\star(t), \ldots, C_{n, d+1}^\star(t), t\geq 0)\to (c_{1}^\star(t), \ldots, c_{d+1}^\star(t), t\geq 0),
\end{align*}
as $n\to\infty$ in $D_{[0, 1]^{d+1}}([0, \infty))$ in the Skorokhod topology, where the latter process is deterministic with initial state $(\infty, 0, \ldots, 0)$ and given by
\begin{align}
	c_{i}^\star(t) = \frac{c^{\infty}(t)^{2-a}}{i!}B_i\left(\risfac{\left(\frac{1}{1-a}\right)}{\bullet}, \risfac{(1-a)}{\bullet}\right)\quad (i\in [d], t\geq 0).
\end{align}
\end{theorem}
\begin{proof}
The process $(C_{n, 1}^\star(t), \ldots, C_{n,d+1}^\star(t), t\geq 0)$ has state space $E_{n, d}^\star=(n^{-1}\mathbb N\cup\{\infty\})^{d+1}$ and initial state $(\infty, 0, \ldots, 0).$ The limiting process $(c_{1}^\star(t), \ldots, c_{d+1}^\star(t), t\geq 0)$ solves the system of ordinary differential equations~\eqref{eq:bss_ode} with initial conditions $c_{1}^\star(0)=\infty$ and $c_{i}^\star(0)=0$ for $i\geq 2.$

We find the solution of this system of ordinary differential equations as follows. First, for some $M\in\mathbb N$ let $(c_{M,1}(t), \ldots, c_{M, d+1}(t), t\geq 0)$ denote the solution of the system of ordinary differential equations~\eqref{eq:bss_ode} but with initial conditions $c_{M, 1}(0)=M$ and $c_{M, i}(0)=0$ for $i\geq 2$. The corresponding generating function
\begin{align}
  \mathscr G_{M}(t, x) \coloneqq \sum_{i\geq 1}c_{M, i}(t)x^{i}
\end{align}
solves the partial differential equation~\eqref{eq:pde} with boundary condition $\mathscr G_{M}(0, x)=Mx.$ Letting $c_{M}(t)\coloneqq \sum_{i\geq 1}c_{M, i}(t),$ hence $c_{M}(t)=\left(M^{a-1}+\frac{\Gamma(a+b)}{(2-a)\Gamma(b)}t\right)^{\frac{1}{a-1}}$ as in the proof of Theorem~\ref{thm:coalescent_limit}, the function
\begin{align}
  g_{M}(t, x) \coloneqq c_{M}(t)-\mathscr G_{M}(t, x)
\end{align}
solves the partial differential equation~\eqref{eq:pde_for_transform} with initial condition $g_{M}(0, x)=M(1-x)$ and therefore
\begin{align}
  g_{M}(t, x) = \left((M(1-x))^{a-1}+\frac{\Gamma(a+b)}{(2-a)\Gamma(b)}t\right)^{\frac{1}{a-1}}.
\end{align}
In complete analogy to Corollary~\ref{cor:limit_of_spectrum} we let
\[ f(x)\coloneqq \left(x +\frac{\Gamma(a+b)}{(2-a)\Gamma(b)}t\right)^{\frac{1}{a-1}}\quad\text{ and }\quad g_{M}(x)\coloneqq (M(1-x))^{a-1},\]
so $g_{M}(t, x)=(f\circ g_{M})(x).$ Since $g^{(j)}_{M}(x)=M^{a-1}(-1)^j\fallfac{(a-1)}{j}(1-x)^{a-1-j},$ we obtain
\begin{align*}
c_i(t) &= -\frac{1}{i!}\partial_x^i\biggr|_{x=0} \mathscr G_{M}(t, x)\\
&= -\frac{1}{i!}\sum_{\pi\in\pset{[i]}}\fallfac{\left(\frac{1}{a-1}\right)}{\card\pi}\left(M^{a-1}+ \frac{\Gamma(a+b)}{(2-a)\Gamma(b)}t\right)^{\frac{1}{a-1}-\card\pi}\prod_{B\in\pi}(-1)^{\card B}\fallfac{(a-1)}{\card B}M^{a-1}\\
&= -\frac{c_{M}(t)}{i!}\sum_{\pi\in\pset{[i]}}(-1)^{\card\pi}\fallfac{\left(\frac{1}{a-1}\right)}{\card\pi}c_{M}(t)^{\card\pi(1-a)}\prod_{B\in\pi}\risfac{(1-a)}{\card B}M^{a-1}\\
&= -\frac{c_{M}(t)}{i!}\sum_{m=1}^i (-1)^{m}\risfac{\left(\frac{1}{1-a}\right)}{m}c_{M}(t)^{m(1-a)} B_{i,m}(\risfac{(1-a)}{\bullet})M^{m(a-1)}\\
&\to \frac{c_{\infty}(t)^{2-a}}{i!}B_i\left(\risfac{\left(\frac{1}{1-a}\right)}{\bullet}, \risfac{(1-a)}{\bullet}\right)
\end{align*}
as $M\to\infty,$ since
\begin{align*}
  (c_{M}(t)/M)^{m(1-a)} = (1+\frac{\Gamma(a+b)t}{(2-a)\Gamma(b)}M^{1-a})^{-m}\to 1
\end{align*}
as $M\to\infty.$
\end{proof}

{\bf Acknowledgement.} Most of the results in this work are taken from the authors' PhD theses. They would like to thank their supervisor Alison Etheridge for her advice and helpful discussions. H.~H.~P.~also thanks Julien Berestycki and Matthias Birkner as well as Martin M\"{o}hle and Elmar Teufl for stimulating discussions, and thankfully acknowledges financial support by the foundation ``Private Stiftung Ewald Marquardt f\"{u}r Wissenschaft und Technik, Kunst und Kultur''.
\bibliographystyle{abbrv}\bibliography{literature}

\end{document}